\documentclass[
  hidelinks
]{siamart220329}

\usepackage{algpseudocode}
\usepackage{amsfonts} 
\usepackage{amsopn}
\usepackage{amssymb}
\usepackage{bm}
\usepackage{booktabs}
\usepackage{caption}
\usepackage{commath}
\usepackage{csvsimple}
\usepackage{enumitem}
\usepackage[T1]{fontenc}
\usepackage{mathtools}
\usepackage{pgfplots}
\usepackage{subcaption}
\usepackage{tikz}
\usepackage{xcolor}

\pgfplotsset{compat=1.11} 

\def\NUMA{%
    NUMA research group, Department of Computer Science, KU Leuven, Leuven, Belgium (\email{\{arne.bouillon,giovanni.samaey,karl.meerbergen\}@kuleuven.be}).
}

\def\submitted{{%
    Submitted to the editors DATE.
    \funding{This work has received funding from the European High-Performance Computing Joint Undertaking (JU) under grant agreement No.\ 955701. The JU receives support from the European Union's Horizon 2020 research and innovation programme and from Belgium, France, Germany, and Switzerland. Karl Meerbergen's work is partly supported by the Research Foundation Flanders (FWO) grants G0B7818N and G088622N, and by the KU Leuven Research Council.}
}}

\makeatletter
\renewcommand*\env@matrix[1][*\c@MaxMatrixCols c]{%
  \hskip -\arraycolsep
  \let\@ifnextchar\new@ifnextchar
  \array{#1}}
\makeatother

\newenvironment{smallarray}[1]
 {\null\,\vcenter\bgroup\scriptsize
  \arraycolsep=.13885em
  \hbox\bgroup$\array{@{}#1@{}}}
 {\endarray$\egroup\egroup\,\null}

\newcommand\kronm{}

\newcommand\kronv\mv

\newcommand\ad\lambda
\newcommand\Ad\Lambda

\newcommand\dfm{{\mathbb F}}

\newcommand\invisibleminus{\phantom{-}}

\newcommand\kron\otimes

\newcommand\en{{\mathrm{end}}}

\newcommand\trsp{{\mathstrut\scriptscriptstyle{\top}}}

\newcommand\iu{{\mathrm{i}\mkern1mu}}

\newcommand\E{{\mathrm e}}

\def\mper{{.}}
\def\mcom{{,}}

\def\sidx{m}
\def\Sidx{M}

\def\tidx{l}
\def\Tidx{L}

\def\iidx{k}

\newcommand\ukus[2]{{#2}}
\def\sz{\ukus{s}{z}}
\def\ou{\ukus{ou}{o}}

\newcommand*\mv[1]{{\bm{{#1}}}}

\DeclareMathOperator\diag{diag}

\definecolor{matlabred}{rgb}{0.6350,0.0780,0.1840}
\definecolor{matlabgreen}{rgb}{0.4660,0.6740,0.1880}
\definecolor{matlabblue}{rgb}{0,0.4470,0.7410}

\algnewcommand{\LineComment}[1]{\Statex \(\triangleright\) #1}
\algnewcommand{\LineCommentWithSkip}[1]{\Statex \hskip 1cm \(\triangleright\) #1}

\usepackage{letltxmacro}
\LetLtxMacro\orgvdots\vdots
\LetLtxMacro\orgddots\ddots

\makeatletter
\DeclareRobustCommand\vdots{%
  \mathpalette\@vdots{}%
}
\newcommand*{\@vdots}[2]{%
  \sbox0{$#1\cdotp\cdotp\cdotp\m@th$}%
  \sbox2{$#1.\m@th$}%
  \vbox{%
    \dimen@=\wd0 %
    \advance\dimen@ -3\ht2 %
    \kern.5\dimen@
    \dimen@=\wd2 %
    \advance\dimen@ -\ht2 %
    \dimen2=\wd0 %
    \advance\dimen2 -\dimen@
    \vbox to \dimen2{%
      \offinterlineskip
      \copy2 \vfill\copy2 \vfill\copy2 %
    }%
  }%
}
\DeclareRobustCommand\ddots{%
  \mathinner{%
    \mathpalette\@ddots{}%
    \mkern\thinmuskip
  }%
}
\newcommand*{\@ddots}[2]{%
  \sbox0{$#1\cdotp\cdotp\cdotp\m@th$}%
  \sbox2{$#1.\m@th$}%
  \vbox{%
    \dimen@=\wd0 %
    \advance\dimen@ -3\ht2 %
    \kern.5\dimen@
    \dimen@=\wd2 %
    \advance\dimen@ -\ht2 %
    \dimen2=\wd0 %
    \advance\dimen2 -\dimen@
    \vbox to \dimen2{%
      \offinterlineskip
      \hbox{$#1\mathpunct{.}\m@th$}%
      \vfill
      \hbox{$#1\mathpunct{\kern\wd2}\mathpunct{.}\m@th$}%
      \vfill
      \hbox{$#1\mathpunct{\kern\wd2}\mathpunct{\kern\wd2}\mathpunct{.}\m@th$}%
    }%
  }%
}
\makeatother

\headers{Generali\sz{}ed time-parallel optimal control}{A.\ Bouillon, G.\ Samaey, and K.\ Meerbergen}
\title{On generali\ukus szed preconditioners for time-parallel parabolic optimal control\thanks{\submitted}}
\author{{Arne Bouillon\thanks{\NUMA}}\and Giovanni Samaey\footnotemark[2]\and Karl Meerbergen\footnotemark[2]}

\begin{document}

\maketitle

\begin{abstract}
    The ParaDiag family of algorithms solves differential equations by using preconditioners that can be inverted in parallel through diagonali\sz{}ation. In the context of optimal control of linear parabolic \textsc{pde}s, the state-of-the-art ParaDiag method is limited to solving self-adjoint problems with a tracking objective. We propose three improvements to the ParaDiag method: the use of alpha-circulant matrices to construct an alternative preconditioner, a generalization of the algorithm for solving non-self-adjoint equations, and the formulation of an algorithm for terminal-cost objectives. We present novel analytic results about the eigenvalues of the preconditioned systems for all discussed ParaDiag algorithms in the case of self-adjoint equations, which proves the fav\ou{}rable properties of the alpha-circulant preconditioner. We use these results to perform a theoretical parallel-scaling analysis of ParaDiag for self-adjoint problems. Numerical tests confirm our findings and suggest that the self-adjoint behavi\ou{}r, which is backed by theory, generali\sz{}es to the non-self-adjoint case. We provide a sequential, open-source reference solver in \textsc{Matlab} for all discussed algorithms.

\end{abstract}



\section{Introduction} \label{sec:intro}
    We are interested in optimal-control problems of the form
\begin{equation} \label{eq:intro:intro:optprob}
    \min_{y,u}J(y,u) \quad \text{such that} \quad y_t = g(y) + u, \quad y(0) = y_\mathrm{init}
\end{equation}
over time $[0, T]$. Here, $y$ represents a space- and time-dependent state variable with initial condition $y_\mathrm{init}$ evolving under the influence of a linear operator $g$, while $u$ is a control input with which $y$ is steered. We want to choose $u$ to minimi\sz{}e $J$, which is either a \emph{tracking} or a \emph{terminal-cost} objective function
\begin{equation} \label{eq:intro:intro:obj}
    J(y,u) = \begin{cases}
        \text{Tracking:} & \frac12\int_0^T\norm{y(t)-y_\mathrm d(t)}_2^2\dif t + \frac\gamma2\int_0^T\norm{u(t)}_2^2\dif t    \mcom\\
        \text{Terminal cost:} & \frac12\norm{y(T)-y_\mathrm{target}}_2^2 + \frac\gamma2\int_0^T\norm{u(t)}_2^2\dif t   \mper
    \end{cases}
\end{equation}
Tracking objectives aim to keep $y$ as close as possible to a trajectory $y_\mathrm d$, while terminal cost only requires the final position to be close to some $y_\mathrm{target}$. The factor $\gamma>0$ regulari\sz{}es the control term and may also model the practical cost of control. We space-discreti\sz{}e this problem (using bold-faced vectors and writing $\mv g(\mv y) = -K\mv y$ for some $K\in\mathbb C^{\Sidx\times\Sidx}$). A solution to \cref{eq:intro:intro:optprob} satisfies the boundary value problem (\textsc{bvp})
\begin{subequations} \label{eq:intro:intro:optsys}
\begin{align}
    &\mv y'(t) = -K\mv y(t) - {\mv\ad(t)}/\gamma, \quad \mv y(0) = \mv{y_\mathrm{init}}    \mcom\\
    &\begin{cases}
        \text{Tracking:} & \mv\ad'(t) = K^*\mv\ad(t) + \mv{y_\mathrm d}(t) - \mv y(t), \quad \mv\ad(T)=\mv0    \mcom\\
        \text{Terminal cost:} & \mv\ad'(t) = K^*\mv\ad(t), \quad \mv\ad(T) = \mv y(T)-\mv{y_\mathrm{target}}    \mcom
    \end{cases}
\end{align}
\end{subequations}
with coupled equations in the state $\mv y(t)$ and the \emph{adjoint} state $\mv\ad(t)\coloneqq-\gamma\mv u(t)$, with one initial and one terminal condition \cite{ganderPARAOPTPararealAlgorithm2020a,wuDiagonalizationbasedParallelintimeAlgorithms2020b,hinzeOptimizationPDEConstraints2009b}. To solve \cref{eq:intro:intro:optsys}, we consider \emph{parallel-in-time} methods for optimal control. These are inspired by time-parallel initial-value problem (\textsc{ivp}) solvers, which overcome the inherently serial nature of time integration. Leveraging these techniques enables the construction of algorithms for optimal control which scale well in parallel when increasing the amount of work in the time dimension.

Some parallel-in-time approaches for the optimal-control problem \cref{eq:intro:intro:optprob} use the \emph{direct-adjoint} optimi\sz{}ation loop \cite{gotschelEfficientParallelinTimeMethod2019a,skeneParallelintimeApproachAccelerating2021a}, where all embedded \textsc{ivp} solves are tackled using time-parallel methods such as \textsc{pfasst} \cite{emmettEfficientParallelTime2012} or the well-known Parareal algorithm \cite{lionsResolutionEDPPar2001a}. Others use the system \cref{eq:intro:intro:optsys}; an example is ParaOpt \cite{ganderPARAOPTPararealAlgorithm2020a}, inspired by Parareal. In this paper, we will expand on the time-parallel algorithm proposed for self-adjoint tracking problems in \cite{wuDiagonalizationbasedParallelintimeAlgorithms2020b}. Belonging to the ParaDiag family \cite{mcdonaldPreconditioningIterativeSolution2018a,ganderConvergenceAnalysisPeriodiclike2019a,liuFastBlockAcirculant2020a,ganderParaDiagParallelintimeAlgorithms2021a,wuParallelInTimeBlockCirculantPreconditioner2020a}, this algorithm constructs a discreti\sz{}ed \emph{all-at-once} system of \cref{eq:intro:intro:optsys} and solves it iteratively, using a preconditioner that is invertible in parallel.

This paper is organi\sz{}ed as follows. We start from the method in \cite{wuDiagonalizationbasedParallelintimeAlgorithms2020b}, whose current preconditioner $P$ is limited as we will see; generali\sz{}ations to new situations are found in \cref{sec:pd-track,sec:pd-tc}. \Cref{sec:pd-track} examines the tracking case, containing an updated \emph{alpha-circulant} preconditioner $P(\alpha)$, analytic expressions for the preconditioned eigenvalues of ParaDiag and an extension of the method to non-self-adjoint problems. \Cref{sec:pd-tc} introduces a novel ParaDiag method for terminal-cost objective functions, again featuring an analytic eigenvalue analysis. In \cref{sec:scale}, we use these results as a theoretical basis to predict weak scalability of both ParaDiag methods for self-adjoint, dissipative equations. The numerical results in \cref{sec:num} confirm this scalability for both self-adjoint and non-self-adjoint equations. In \cref{sec:concl}, we conclude and propose further research directions.

As a final note, we mention the very recent paper \cite{Lin_2022}, which also constructs alpha-circulant preconditioners for self-adjoint tracking problems. Our approach is very different, offering a more direct generali\sz{}ation of \cite{wuDiagonalizationbasedParallelintimeAlgorithms2020b}. The analysis presented here results in exact analytical eigenvalues (both for our method and for \cite{wuDiagonalizationbasedParallelintimeAlgorithms2020b}) instead of a bound and our method straightforwardly generali\sz{}es to non-self-adjoint equations.


\section{ParaDiag for tracking objectives} \label{sec:pd-track}
    This section considers the tracking objective in \cref{eq:intro:intro:obj}, for which a ParaDiag procedure (limited to problems with self-adjoint $K=K^*$) is described in \cite{wuDiagonalizationbasedParallelintimeAlgorithms2020b}. We review this method in \cref{sec:pd-track:existing}. Subsequently, \cref{sec:pd-track:alpha} looks at the limiting case $T\rightarrow0$, in which ParaDiag is discovered to lack robustness. We counteract this with an improvement to the preconditioner, using novel analytic results in \cref{sec:pd-track:anal,sec:pd-track:interp} to prove its more fav\ou{}rable properties. \Cref{sec:pd-track:gen} concludes by proposing a generali\sz{}ation to problems where $K\ne K^*$.

For ease of exposition, we use an implicit-Euler time discreti\sz{}ation with time step $\tau$ throughout this paper, although other discreti\sz{}ations can be treated similarly.

    \subsection{Existing method} \label{sec:pd-track:existing}
    The existing algorithm requires a self-adjoint $K$ -- that is, $K = K^*$. The all-at-once system for \cref{eq:intro:intro:optsys} then reads \cite{wuDiagonalizationbasedParallelintimeAlgorithms2020b}
\begin{equation} \label{eq:pd-track:existing:aao}
    \kronm A\begin{litmat}
        \kronv y\\ \kronv\ad\\
    \end{litmat} \coloneqq \left(\begin{litmat}
        B & \frac{\tau I_t}\gamma\\
        -\tau I_t & B^\trsp\\
    \end{litmat} \kron I_x + \tau\begin{litmat}
        I_t\\&I_t\\
    \end{litmat} \kron K \right)\begin{litmat}
        \kronv y\\ \kronv\ad\\
    \end{litmat} = \begin{litmat}
        \kronv{b_1}\\ \kronv{b_2}\\
    \end{litmat}    \mcom
\end{equation}
where $I_t$ and $I_x$ are identity matrices in the context of time and space and
\begin{equation} \label{eq:pd-track:existing:Bbb}
    B = \left[\begin{smallmatrix}
        1 &\\-1 & 1 &\\
        & \ddots & \ddots\\
        && -1 & 1
    \end{smallmatrix}\right], \quad \kronv{b_1} = \begin{litmat}\mv{y_\mathrm{init}}^\trsp & 0 & \ldots & 0\end{litmat}^\trsp, \quad \text{and} \quad \kronv{b_2} = -\tau\kronv{y_\mathrm d}    \mper
\end{equation}

In \cref{eq:pd-track:existing:aao}, we grouped the discreti\sz{}ed unknowns $\kronv y = \bigl[\begin{smallmatrix}
    \mv y_1^\trsp & \mv y_2^\trsp & \cdots & \mv y_{\Tidx-1}^\trsp
\end{smallmatrix}\bigr]^\trsp$ and $\kronv \ad = \bigl[\begin{smallmatrix}
    \mv \ad_1^\trsp & \mv \ad_2^\trsp & \cdots & \mv \ad_{\Tidx-1}^\trsp
\end{smallmatrix}\bigr]^\trsp$, where $\mv y_\tidx$ and $\mv\ad_\tidx$ are approximations to $\mv y(t=\tidx\tau)$ and $\mv\ad(t=\tidx\tau)$. The relation of $\kronv{y_\mathrm d}$ to $\mv{y_\mathrm d}(t)$ is analogous. The discreti\sz{}ations $\mv y_0$ and $\mv\ad_\Tidx$ are known from \cref{eq:intro:intro:optsys}, while $\mv y_\Tidx = (I_x+\tau K)^{-1}(\mv y_{\Tidx-1} - \frac{\tau\mv\ad_\Tidx}\gamma)$ and $\mv\ad_0 = (I_x+\tau K)^{-1}(\mv\ad_1+\tau(\mv y_0-\mv y_{\mv{\mathrm d}, 0}))$ each only appear in one equation. We introduce $\widehat\Tidx\coloneqq\Tidx-1$ such that $B$ is $\widehat\Tidx\times\widehat\Tidx$.

Next, a rescaling is applied\footnote{In \cite{wuDiagonalizationbasedParallelintimeAlgorithms2020b}, $\kronv y$ and $\kronv{b_1}$ are rescaled, but this is completely equivalent.}: $\kronv{\widehat\ad}\coloneqq\kronv\ad/\sqrt\gamma$ and $\kronv{\widehat b_2} \coloneqq \kronv{b_2}/\sqrt\gamma$. We iteratively solve
\begin{equation} \label{eq:pd-track:existing:aao-rescaled}
    \kronm{\widehat A}\begin{litmat}
        \kronv y\\ \kronv{\widehat\ad}\\
    \end{litmat} \coloneqq \Bigg(\begin{litmat}
        B & \frac{\tau I_t}{\sqrt\gamma}\\
        -\frac{\tau I_t}{\sqrt\gamma} & B^\trsp\\
    \end{litmat} \kron I_x + \tau\begin{litmat}
        I_t\\&I_t\\
    \end{litmat} \kron K \Bigg)\begin{litmat}
        \kronv y\\ \kronv{\widehat\ad}\\
    \end{litmat} = \begin{litmat}
        \kronv{b_1}\\ \kronv{\widehat b_2}\\
    \end{litmat}
\end{equation}
for $\kronv y$ and $\kronv{\widehat\ad}$ using e.g.\ \textsc{gmres} \cite{saadGMRESGeneralizedMinimal1986a}, with a preconditioner we will later invert in parallel:
\begin{equation} \label{eq:pd-track:existing:P}
    \kronm P = \begin{litmat}
        C & \frac{\tau I_t}{\sqrt\gamma}\\
        -\frac{\tau I_t}{\sqrt\gamma} & C^\trsp\\
    \end{litmat} \kron I_x + \tau\begin{litmat}
        I_t\\&I_t\\
    \end{litmat} \kron K \quad \text{with} \quad C = \left[\begin{smallmatrix}
        1 & & & -1\\-1 & 1 &\\
        & \ddots & \ddots\\
        && -1 & 1
    \end{smallmatrix}\right]    \mper
\end{equation}
As can be found in \cite{biniNumericalMethodsStructured2005a}, any \emph{circulant} matrix such as $C\in\mathbb R^{\widehat\Tidx\times\widehat\Tidx}$ diagonali\sz{}es as
\begin{equation} \label{eq:pd-track:existing:C-fact}
C=\dfm ^*D\dfm  \quad \text{with} \quad D=\diag(\sqrt{\widehat\Tidx}\dfm \mv c_1), \quad \text{where $\mv c_1$ is $C$'s first column}    \mper
\end{equation}
Here, $\dfm  = \{\E^{2\pi\iu jk/\widehat\Tidx}/\sqrt{\widehat\Tidx}\}_{j,k=0}^{\widehat\Tidx-1}$ is the discrete Fourier matrix. The work \cite{wuDiagonalizationbasedParallelintimeAlgorithms2020b} factori\sz{}es
\begin{equation} \label{eq:pd-track:existing:P-fac}
    \kronm P = \left(
        \begin{litmat}
            \dfm ^*\\&\dfm ^*\\
        \end{litmat} \kron I_x
    \right)\Bigg(
        \begin{litmat}
            D & \frac{\tau I_t}{\sqrt\gamma}\\
            -\frac{\tau I_t}{\sqrt\gamma} & D^*\\
        \end{litmat} \kron I_x + \tau\begin{litmat}
            I_t\\&I_t\\
        \end{litmat} \kron K
    \Bigg)\left(
        \begin{litmat}
            \dfm \\&\dfm \\
        \end{litmat} \kron I_x
    \right)
\end{equation}
and argues for a diagonali\sz{}ation $\bigl[\begin{smallmatrix}D & {\tau I_t}/{\sqrt\gamma}\\
    -\tau I_t/{\sqrt\gamma} & D^*\end{smallmatrix}\bigr] = WHW^{-1}$ with $W=\bigl[\begin{smallmatrix}I_t & S_2\\S_1 & I_t\end{smallmatrix}\bigr]$, where $H$ and $S_{\{1,2\}}$ are diagonal (see \cite{wuDiagonalizationbasedParallelintimeAlgorithms2020b} for details). Defining $V \coloneqq \bigl[\begin{smallmatrix}
    \dfm ^*\\&\dfm ^*\\
\end{smallmatrix}\bigr]W$,
\begin{equation} \label{eq:pd-track:existing:P-fac-further}
    \kronm P^{-1} = (V \kron I_x)(
        H \kron I_x + \tau I_t \kron K
    )^{-1}(V^{-1} \kron I_x)    \mper
\end{equation}
\Cref{alg:pd-track:existing} summari\sz{}es how to solve \cref{eq:pd-track:existing:aao}, using a parallel multiplication by $\kronm P^{-1}$.

\begin{algorithm}[H]
    \caption{ParaDiag for solving the tracking problem \cref{eq:pd-track:existing:aao}, based on \cite{wuDiagonalizationbasedParallelintimeAlgorithms2020b}} \label{alg:pd-track:existing}
    \begin{tabular}{rl}
        \textbf{Input:}  &Vectors $\kronv{b_1}$ and $\kronv{b_2}$ defined by \cref{eq:pd-track:existing:Bbb}\\
                            &\emph{Self-adjoint} matrix $K$ characterising the problem by \cref{eq:intro:intro:optsys}\\
                            &Matrices $H$ and $W$ following from the time discreti\sz{}ation\\
        \textbf{Output:} &The vectors $\kronv y$ and $\kronv\ad=\sqrt\gamma\kronv{\widehat\ad}$ that solve \cref{eq:pd-track:existing:aao}\\
    \end{tabular}
    \begin{algorithmic}[1]
        \State Rescale $\kronv{\widehat b_2} =  \kronv{b_2}/\sqrt\gamma$.
        \State Solve \cref{eq:pd-track:existing:aao-rescaled} for $\kronv y$ and $\kronv{\widehat\ad}$ using an iterative method, with preconditioner $\kronm P$ from \cref{eq:pd-track:existing:P}. When asked to compute $\bigl[\begin{smallmatrix}\kronv x\\ \kronv z\end{smallmatrix}\bigr] = \kronm P^{-1}\bigl[\begin{smallmatrix}\kronv v\\ \kronv w\end{smallmatrix}\bigr]$:
            \State \hskip1cm Calculate $\kronv{r_1} \coloneqq (\dfm  \kron I_x)\kronv{v}$ and $\kronv{s_1} \coloneqq (\dfm  \kron I_x)\kronv{w}$ with the (parallel) \textsc{fft}.
            \State \hskip1cm Calculate $\kronv{q_2}\coloneqq(W^{-1}\kron I_x)\bigl[\begin{smallmatrix}
                \kronv{r_1}\\\kronv{s_1}
            \end{smallmatrix}\bigr]$.
            \State \hskip1cm For $\tidx=\{1, \ldots, 2\widehat\Tidx\}$, solve (in parallel)
            \begin{equation}
                \mv q_{\mv{3},\tidx}\coloneqq (h_{\tidx,\tidx}I_x+\tau K)^{-1}\mv q_{\mv{2},\tidx}
            \end{equation}
            \hskip1cm and partition the variables as $\bigl[\begin{smallmatrix}\kronv r_3\\ \kronv s_3\end{smallmatrix}\bigr] \coloneqq \kronv{q_3}$.
            \State \hskip1cm Calculate $\bigl[\begin{smallmatrix}\kronv r_4\\ \kronv s_4\end{smallmatrix}\bigr]\coloneqq(W\kron I_x)\bigl[\begin{smallmatrix}\kronv r_3\\ \kronv s_3\end{smallmatrix}\bigr]$.
            \State \hskip1cm Calculate $\kronv{x} = (\dfm ^* \kron I_x)\kronv{r_4}$ and $\kronv{z} = (\dfm ^* \kron I_x)\kronv{s_4}$ with the (parallel) \textsc{fft}.
    \end{algorithmic}
\end{algorithm}

    \subsection{The small-\texorpdfstring{$T$}{T} limit and alpha-circulants} \label{sec:pd-track:alpha}
    When using iterative linear-system solvers, convergence speed often depends substantially on the distribution of the eigenvalues of the preconditioned matrix \cite{trefethenNumericalLinearAlgebra1997a} -- in our case, of $\kronm P^{-1}\kronm{\widehat A}$ (while there are exceptions such as \textsc{cgn}, which relies on singular values instead, the rest of this paper will assume the solver behavi\ou{}r is mainly determined by the eigenvalues). Specifically, eigenvalues that are clustered together and lie far enough from $0$ are beneficial. While we stress that convergence is not exclusively determined by eigenvalues (for an extreme example, see \cite{greenbaumAnyNonincreasingConvergence1996}), they play an important role, and making an educated guess about convergence based on them is common \cite{pearsonRegularizationRobustPreconditionersTimeDependent2012}. In particular, \cite{wuDiagonalizationbasedParallelintimeAlgorithms2020b} performed an empirical eigenvalue study for the method in \cref{sec:pd-track:existing} and compared the \textsc{gmres} and BiCGStab \cite{vandervorstBiCGSTABFastSmoothly1992} iterative solvers, showing the former to be faster. To illustrate eigenvalues' importance, we choose \textsc{gmres} and follow \cite{wuDiagonalizationbasedParallelintimeAlgorithms2020b} in considering the discreti\sz{}ed Laplacian on spatial domain $\Omega=[0,1]$ with isolated boundary,
\begin{equation} \label{eq:pd-track:alpha:Kexample}
    K = \frac1{\Delta\!x^2}\left[\rule{0cm}{.8cm}\right.\begin{smallmatrix}
        1 & -1\\
        -1 & 2 & \ddots\\
        & \ddots & \ddots & \ddots\\
        &&\ddots & 2 & -1\\
        &&&-1&1
    \end{smallmatrix}\left.\rule{0cm}{.8cm}\right] \in \mathbb R^{M\times M}    \mcom
\end{equation}
where $\Delta\!x=1/M$. We use $M=16$, $L=128$, and $\gamma=10^{-5}$ as in \cite{wuDiagonalizationbasedParallelintimeAlgorithms2020b} and set $y_\mathrm d(x,t) = y_\mathrm{init}(x) = \exp(-100(x-0.5)^2)$ -- these do not impact the preconditioned eigenvalues, but may still influence the iteration count. We study this in two regimes. \Cref{fig:pd-track:eigs-1:1} uses \cite{wuDiagonalizationbasedParallelintimeAlgorithms2020b}'s time horizon $T=1$. The eigenvalues cluster around unity and the \textsc{gmres}\footnote{We use a relative \textsc{gmres} tolerance of $10^{-6}$, which is \textsc{Matlab}'s default, throughout this paper.} iteration count $k_g$ is low. When reducing the time interval by setting $T=10^{-4}$, however, \cref{fig:pd-track:eigs-0.0001:1} reveals large variations in the eigenvalues and an increased iteration count. \Cref{sec:pd-track:anal} will study this difference analytically.

\begin{figure}
    \centering
    \begin{subfigure}[b]{.4\textwidth}
        \includegraphics[width=\textwidth]{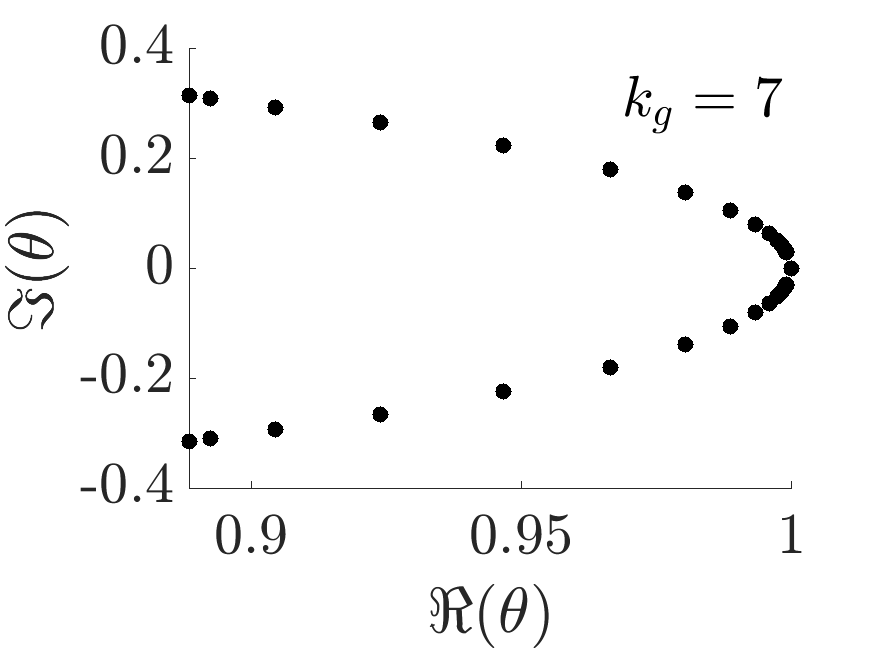}
        \caption{
            $T=1$
            \vspace{-.5cm}
        }
        \label{fig:pd-track:eigs-1:1}
    \end{subfigure}
    \hfill
    \begin{subfigure}[b]{.4\textwidth}
        \includegraphics[width=\textwidth]{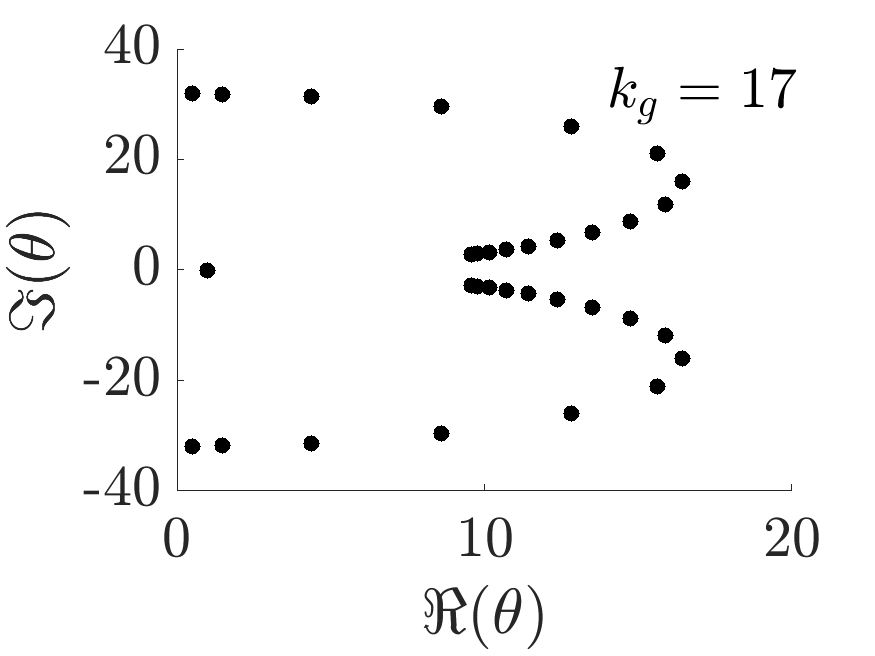}
        \caption{
            $T=10^{-4}$
            \vspace{-.5cm}
        }
        \label{fig:pd-track:eigs-0.0001:1}
    \end{subfigure}
    \caption{
        Eigenvalues $\theta$ of $\kronm P^{-1}\kronm{\widehat A}$ and iteration count $k_g$ of ParaDiag for the example in \cref{sec:pd-track:alpha}, using \textsc{gmres} with relative tolerance $10^{-6}$. \Cref{fig:pd-track:eigs-1:1} mimics \cite[Figure 6]{wuDiagonalizationbasedParallelintimeAlgorithms2020b}, but \cref{fig:pd-track:eigs-0.0001:1} discovers issues when $T$ is small.
        \vspace{-.7cm}
    }
    \label{fig:pd-track:eigs-1}
\end{figure}

We first propose an altered preconditioner $\kronm P(\alpha)$ with a parameter $\alpha\in\mathbb C$. Let
\begin{equation} \label{eq:pd-track:alpha:Palpha}
    \kronm P(\alpha) = \begin{litmat}
        C(\alpha) & \tau\frac{I_t}{\sqrt\gamma}\\
        -\tau\frac{I_t}{\sqrt\gamma} & C(\alpha)^*\\
    \end{litmat} \kron I_x + \tau\begin{litmat}
        I_t\\&I_t\\
    \end{litmat} \kron K \quad \text{with} \quad C(\alpha) = \left[\begin{smallmatrix}
        1 & & & -\alpha\\-1 & 1 &\\
        & \ddots & \ddots\\
        && -1 & 1
    \end{smallmatrix}\right]
\end{equation}
where $C(\alpha)$ is not circulant, but \emph{alpha-circulant} (circulant except that super-diagonal entries have been multiplied by some $\alpha\ne0$). Alpha-circulants diagonali\sz{}e as \cite{biniNumericalMethodsStructured2005a}
\begin{equation} \label{eq:pd-track:alpha:Calpha-fact}
    C(\alpha) = VD(\alpha)V^{-1} \quad \text{with}  \quad V=\Gamma_\alpha^{-1}\dfm ^* \quad \text{and} \quad D(\alpha)=\diag(\sqrt{\widehat\Tidx}\dfm \Gamma_\alpha\mv c_1)    \mcom
\end{equation}
where $\Gamma_\alpha=\diag(1, \alpha^{1/\widehat\Tidx}, \cdots, \alpha^{(\widehat\Tidx-1)/\widehat\Tidx})$. As before, $\mv c_1$ (which is independent of $\alpha$) is $C(\alpha)$'s first column. The idea of using alpha-circulants occurs in the \textsc{ivp} literature \cite{ganderConvergenceAnalysisPeriodiclike2019a,liuFastBlockAcirculant2020a}, but while \textsc{ivp}s can use $\alpha\approx 0$ such that $\kronm P(\alpha) \approx \kronm{\widehat A}$, our $\kronm P(\alpha)$ comes with severe limitations. To successfully perform a factori\sz{}ation like \cref{eq:pd-track:existing:P-fac}, the matrices~$C(\alpha)$ and~$C(\alpha)^*$ must be simultaneously diagonali\sz{}able. This is only the case when
\begin{equation} \label{eq:pd-track:alpha:alphaunitcirc}
    \Gamma_\alpha^{-1} = \Gamma_\alpha^* \Leftrightarrow \abs\alpha=1    \mper
\end{equation}
Contrary to the \textsc{ivp} situation, under the constraint \cref{eq:pd-track:alpha:alphaunitcirc} it is far less clear that setting $\alpha\ne1$ is advantageous. We choose\footnote{$(-1)$-circulants have also been called \emph{skew-circulant} or \emph{negacyclic} matrices \cite{davisCirculantMatrices1979}.} $\alpha=-1$ to reiterate our previous experiment. \Cref{fig:pd-track:eigs--1} shows that $\kronm P(-1)$ does not display the same defects for small $T$ as did $\kronm P = \kronm P(1)$: the case $T=1$ looks identical, but for $T=10^{-4}$ the eigenvalues cluster instead of dispersing and the \textsc{gmres} iteration count remains low. Though not shown here, complex values of $\alpha$ on the unit circle far enough from $1$ yield similar results. In the next subsection, we unravel the behavi\ou{}r for real $\alpha$s analytically. Even if such small $T$s rarely show up in practice, we will see that $\alpha=-1$ allows us to formulate strong analytical results about the preconditioners by covering this edge case.

\begin{figure}
    \centering
    \begin{subfigure}[b]{.45\textwidth}
        \includegraphics[width=\textwidth]{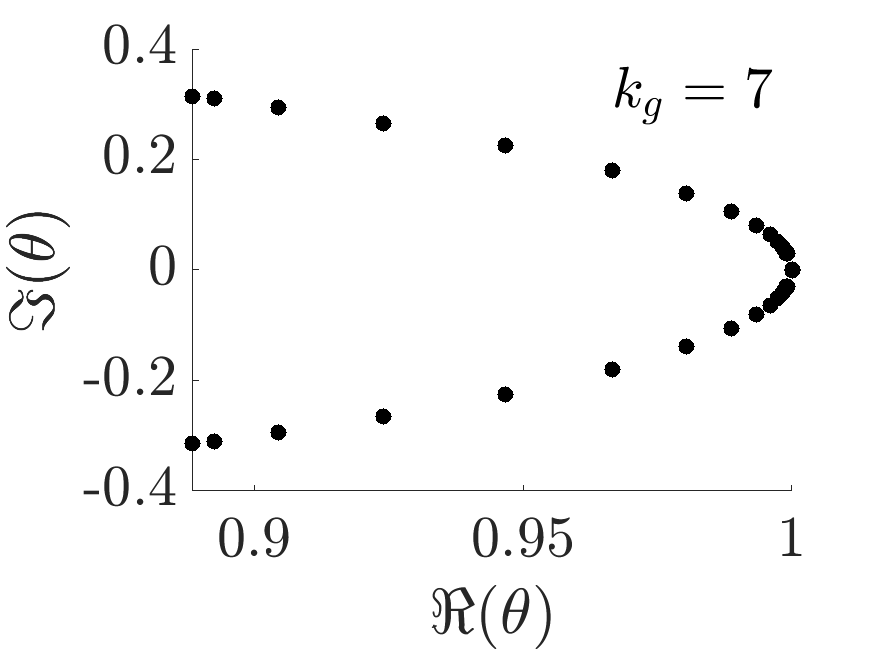}
        \caption{
            $T=1$
        }
        \label{fig:pd-track:eigs-1:-1}
    \end{subfigure}
    \hfill
    \begin{subfigure}[b]{.45\textwidth}
        \includegraphics[width=\textwidth]{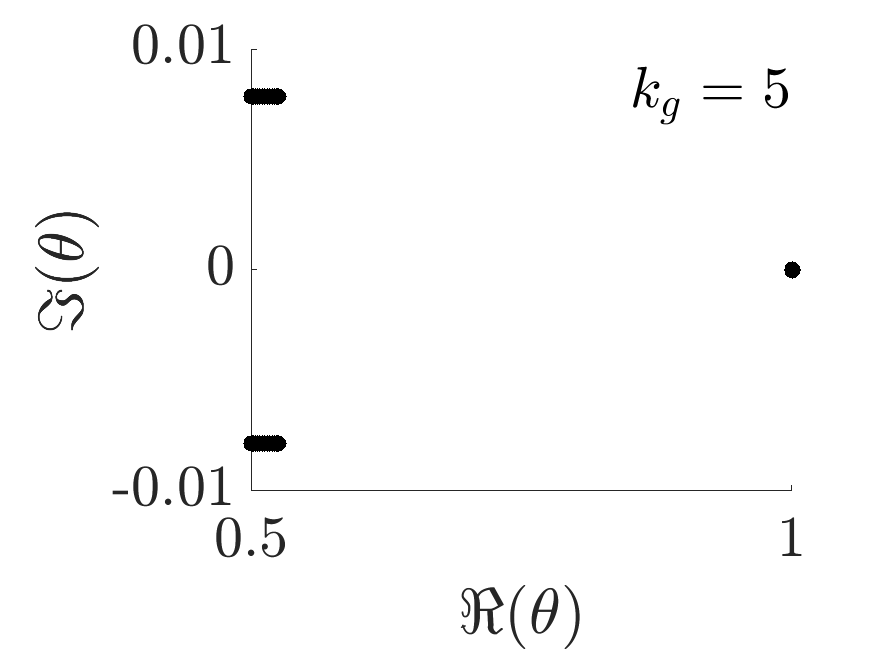}
        \caption{
            $T=10^{-4}$
        }
        \label{fig:pd-track:eigs-0.0001:-1}
    \end{subfigure}
    \caption{
        Eigenvalues $\theta$ of $\kronm P(-1)^{-1}\kronm{\widehat A}$ and iteration count $k_g$ of ParaDiag for the example in \cref{sec:pd-track:alpha}, using \textsc{gmres} with relative tolerance $10^{-6}$
    }
    \label{fig:pd-track:eigs--1}
\end{figure}

    \subsection{Analytic eigenvalue expressions} \label{sec:pd-track:anal}
    We study the preconditioner $\kronm P(\alpha)$ for two significant cases: $\alpha=\pm1$. We start by rescaling the system \cref{eq:pd-track:existing:aao-rescaled} by multiplying by $I_t\kron(I_x + \tau K)^{-1}$. This yields the system
\begin{equation} \label{eq:pd-track:anal:resc}
    \kronm{\widehat A_\mathrm p}\begin{litmat}
        \kronv y\\ \kronv{\widehat\ad}\\
    \end{litmat} \coloneqq \left[\begin{smallarray}{cccc|cccc}
        I_x & & & & \invisibleminus\Psi\\
        -\Phi & I_x & & & & \invisibleminus\Psi\\
        & \ddots & \ddots & & & & \ddots\\
        & & -\Phi & I_x & & & & \invisibleminus\Psi\\
        \cmidrule(lr){1-4}\cmidrule(lr){5-8}
        -\Psi & & & & I_x & -\Phi\\
        & \ddots & & & & \ddots & \ddots\\
        & & -\Psi & & & & I_x & -\Phi\\
        & & & -\Psi & & & & I_x\\
    \end{smallarray}\right]\begin{litmat}
        \kronv y\\ \kronv{\widehat\ad}\\
    \end{litmat} = \kronv{\widehat b_\mathrm p}    \mcom
\end{equation}
with $\kronv{\widehat b_\mathrm p}$ a rescaled version of $\kronv{\widehat b}$. For this implicit-Euler discreti\sz{}ation, $\Phi=(I_x+\tau K)^{-1}$ and $\Psi = \tau/\sqrt\gamma(I_x+\tau K)^{-1}$. We perform the same rescaling on $\kronm P(\alpha)$ and name the result $\kronm P_\mathrm p(\alpha)$. The preconditioned matrix $\kronm P(\alpha)^{-1}\kronm{\widehat A}$ is then equal to $\kronm P_\mathrm p(\alpha)^{-1}\kronm{\widehat A}_\mathrm p$.

We make use of the self-adjointness of $K$ (and, thus, of $\Phi$ and $\Psi$) to drastically reduce the number of parameters that determine the eigenvalues of $\kronm P_\mathrm p(\alpha)^{-1}\kronm{\widehat A}_\mathrm p$. Write $K=V^{-1}\Sigma V$ with diagonal $\Sigma$. $\Phi$ and $\Psi$ are also diagonali\sz{}able by $V$, so $\kronm P_\mathrm p(\alpha)^{-1}\kronm{\widehat A}_\mathrm p$ can be decomposed into \emph{separate} versions for each of $K$'s eigenvalues $\sigma \in \diag\Sigma$. The eigenvalues of $\kronm P_\mathrm p(\alpha)^{-1}\kronm{\widehat A}_\mathrm p$ are the union of those of the matrices
\begin{equation}
    P_{\sigma}(\alpha)^{-1}\widehat A_{\sigma} \coloneqq \left[\begin{smallarray}{cccc|cccc}
        1 & & & -\alpha\varphi & \invisibleminus\psi\\
        -\varphi & 1 & & & & \invisibleminus\psi\\
        & \ddots & \ddots & & & & \ddots\\
        & & -\varphi & 1 & & & & \invisibleminus\psi\\
        \cmidrule(lr){1-4}\cmidrule(lr){5-8}
        -\psi & & & & 1 & -\varphi\\
        & \ddots & & & & \ddots & \ddots\\
        & & -\psi & & & & 1 & -\varphi\\
        & & & -\psi & -\alpha\varphi & & & 1\\
    \end{smallarray}\right]^{-1}\left[\begin{smallarray}{cccc|cccc}
        1 & & & & \invisibleminus\psi\\
        -\varphi & 1 & & & & \invisibleminus\psi\\
        & \ddots & \ddots & & & & \ddots\\
        & & -\varphi & 1 & & & & \invisibleminus\psi\\
        \cmidrule(lr){1-4}\cmidrule(lr){5-8}
        -\psi & & & & 1 & -\varphi\\
        & \ddots & & & & \ddots & \ddots\\
        & & -\psi & & & & 1 & -\varphi\\
        & & & -\psi & & & & 1\\
    \end{smallarray}\right]
\end{equation}
for all $\sigma$, with $\varphi=(1+\tau\sigma)^{-1}$ and $\psi=\frac\tau{\sqrt\gamma}\varphi$. Then, by writing $\widehat\sigma\coloneqq\tau\sigma$ and $\widehat\gamma\coloneqq\frac\tau{\sqrt\gamma}$,
\begin{equation} \label{eq:pd-track:anal:phipsi}
    \varphi = (1 + \widehat\sigma)^{-1} \quad \text{and} \quad \psi = \widehat\gamma(1+\widehat\sigma)^{-1}    \mcom
\end{equation}
which eliminates $\tau$ as an independent variable. The only parameters left that are relevant to the eigenvalues of $P(\alpha)^{-1}\widehat A$ are the different $\widehat\sigma$s (time step--rescaled eigenvalues of $K$), $\widehat\gamma$ (a time step--rescaled value indicating ``how much control'' is present) and $L$ (the number of time steps, dictating the size $\widehat\Tidx=\Tidx-1$ of the matrices).

In summary, each eigenvalue $\sigma$ of $K$ defines a preconditioned matrix $P_\sigma(\alpha)^{-1}\widehat A_\sigma$, whose eigenvalues are also eigenvalues of $\kronm P(\alpha)^{-1}\kronm{\widehat A}$. These eigenvalues influence the iterative solver's convergence, as explained in \cref{sec:pd-track:alpha}. The matrices $P_\sigma(\alpha)^{-1}\widehat A_\sigma$ can be rewritten as the identity plus a low-rank term:

\begin{equation} \label{eq:pd-track:anal:defR}
\begin{aligned}
    P_\sigma(\alpha)^{-1} \widehat A_\sigma &= P_\sigma(\alpha)^{-1}(P_\sigma(\alpha) + (\widehat A_\sigma - P_\sigma(\alpha)))\\
    &= I_t + P_\sigma(\alpha)^{-1}(\widehat A_\sigma - P_\sigma(\alpha)) \eqqcolon I_t + P_\sigma(\alpha)^{-1}R_\sigma    \mcom
\end{aligned}
\end{equation}
from which it follows that the eigenvalues $\theta$ of $P_\sigma(\alpha)^{-1}\widehat A_\sigma$ are equal to one plus the eigenvalues $\omega$ of $P_\sigma(\alpha)^{-1}R_\sigma$. The latter are characteri\sz{}ed by \cref{thm:pd-track:eigs}.

\begin{theorem} \label{thm:pd-track:eigs}
    Let $\widehat\Tidx>3$, $\alpha=\pm1$ and $\varphi,\psi\in\mathbb R\backslash\{0\}$. The $2\widehat\Tidx\times2\widehat\Tidx$ matrix
    \begin{equation} \label{eq:thm:pd-track:eigs:M}
    \begin{aligned}
        M &= \overbrace{\left[\begin{smallarray}{cccc|cccc}
            1&&&-\alpha\varphi&\invisibleminus\psi\\
            -\varphi & 1 &&&&\invisibleminus\psi\\
            &\ddots&\ddots&&&&\ddots\\
            &&-\varphi&1&&&&\invisibleminus\psi\\
            \cmidrule(lr){1-4}\cmidrule(lr){5-8}
            -\psi&&&&1&-\varphi\\
            &\ddots&&&&\ddots&\ddots\\
            &&-\psi&&&&1&-\varphi\\
            &&&-\psi&-\alpha\varphi&&&1\\
        \end{smallarray}\right]^{-1}}^{=P_\sigma(\alpha)^{-1}}\overbrace{
            \left[\begin{smallarray}{cccc|cccc}
                \phantom{1}&&&\phantom{-}\alpha\varphi&\phantom{\invisibleminus\psi}\\
                \phantom{-\varphi} & \phantom{1}&&&&\phantom{\invisibleminus\psi}\\
                &\phantom{\ddots}&\phantom{\ddots}&&&&\phantom{\ddots}\\
                &&\phantom{-\varphi}&\phantom{1}&&&&\phantom{\invisibleminus\psi}\\
                \cmidrule(lr){1-4}\cmidrule(lr){5-8}
                \phantom{-\psi}&&&&\phantom{1}&\phantom{-\varphi}\\
                &\phantom{\ddots}&&&&\phantom{\ddots}&\phantom{\ddots}\\
                &&\phantom{-\psi}&&&&\phantom{1}&\phantom{-\varphi}\\
                &&&\phantom{-\psi}&\phantom{-}\alpha\varphi&&&\phantom{1}\\
            \end{smallarray}\right]}^{=R_\sigma}
    \end{aligned}
    \end{equation}
    has only two potentially non-zero eigenvalues $\omega_{\{1,2\}}$. Specifically,
    \begin{equation} \label{eq:thm:pd-track:eigs:eigs}
        \omega_{\{1,2\}} = \frac1{z_2-z_1}\bigg(
            \frac{z_1-\varphi\pm\psi\iu}{1-\alpha z_1^{\widehat\Tidx}}-\frac{z_2-\varphi\pm\psi\iu}{1-\alpha z_2^{\widehat\Tidx}}
        \bigg)
    \end{equation}
    where
    \begin{equation} \label{eq:thm:pd-track:eigs:z}
        z_{\{1,2\}}=\frac{1+\varphi^2+\psi^2\pm\sqrt{(1+\varphi^2+\psi^2)^2-4\varphi^2}}{2\varphi}    \mper
    \end{equation}
\end{theorem}
\begin{proof}
    The proof of this theorem is given in \cref{sec:apdx-pdte}.
\end{proof}

\clearpage
\begin{corollary} \label{cor:pd-track:anal:thetas}
    Consider a tracking-type all-at-once system with implicit-Euler time discreti\sz{}ation \cref{eq:pd-track:existing:aao-rescaled} where $\Tidx>4$ and $K$ is self-adjoint with eigenvalues $\left\{\sigma_\sidx\right\}_{\sidx=1}^{\Sidx}$. When using ParaDiag with preconditioner $\kronm P(\alpha=\pm 1)$ (see \cref{eq:pd-track:alpha:Palpha}), the eigenvalues of the preconditioned system matrix $\kronm P(\alpha)^{-1}\kronm{\widehat A}$ are all either unity or equal to
    \begin{equation}
        \theta_{\sidx,\{1,2\}} = 1 + \omega_{\sidx,\{1,2\}}
    \end{equation}
    where $\omega_{\sidx,\{1,2\}}$ are given by the formula \cref{eq:thm:pd-track:eigs:eigs}, having filled in $\varphi = (1+\tau\sigma_\sidx)^{-1}$ and $\psi = \frac\tau{\sqrt\gamma}(1+\tau\sigma_\sidx)^{-1}$.
\end{corollary}

\begin{corollary} \label{cor:pd-track:anal:semidisc}
    Consider a tracking-type all-at-once system with implicit-Euler time discreti\sz{}ation \cref{eq:pd-track:existing:aao-rescaled} where $\Tidx>4$ and $K$ is self-adjoint. Denote by $\mathcal D_{0.5,+}$ the right half of a dis\ukus{c}{k} in the complex plane, cent\ukus{re}{ere}d at $0.5$ and with radius $0.5$. When using ParaDiag with preconditioner $\kronm P(\alpha=-1)$ (see \cref{eq:pd-track:alpha:Palpha}), if $0<\varphi<1$ for all $\varphi$ (which occurs whenever $K$ is positive definite), all eigenvalues of the preconditioned matrix $\kronm P(\alpha)^{-1}\kronm{\widehat A}$ lie within $\mathcal D_{0.5,+}$.
\end{corollary}
\begin{proof}
    \Cref{lmm:apdx-proof:pd-track:thetas}(c) shows that the real parts of these eigenvalues are larger than $0.5$, while \cref{lmm:apdx-proof:pd-track:thetas}(d) proves that their distances from the point $0.5$ are less than $0.5$. Together, these bounds delineate the region $\mathcal D_{0.5,+}$.
\end{proof}

    \subsection{Interpreting the eigenvalue results} \label{sec:pd-track:interp}
    It is possible to visuali\sz{}e \cref{cor:pd-track:anal:thetas} to gain more insight into how the two preconditioners (that is, $\alpha=1$ and $\alpha=-1$) perform, as well as how they compare. Recall from \cref{sec:pd-track:anal} that every eigenvalue $\sigma_\sidx$ of $K$ corresponds to two non-unity eigenvalues $\theta_{\sidx,\{1,2\}}$ of the preconditioned system matrix, which are complex conjugates of each other. \Cref{fig:pd-track:eigs} plots the $\theta$ with positive imaginary part for the cases $\alpha=\pm1$ in two distinct ways. This section considers dissipative problems, where $K$ is positive definite and hence $\sigma>0$.

\Cref{fig:pd-track:plane:m1,fig:pd-track:plane:p1} are based on the view that $\widehat\gamma$ is typically known (one can set the regulari\sz{}ation parameter $\gamma$ and the time step $\tau$), while the eigenvalues $\sigma$ (and thus $\widehat\sigma=\tau\sigma$) could lie anywhere. For each $\widehat\gamma$ value, these figures mark the preconditioned eigenvalues for a whole range of $\widehat\sigma>0$ options. We can see \cref{cor:pd-track:anal:semidisc} in action: for $\alpha=-1$ the eigenvalues lie inside $\mathcal D_{0.5,+}$ while, for $\alpha=1$, they lie outside it. \Cref{fig:pd-track:eigs-abs:m1,fig:pd-track:eigs-abs:p1} add $\widehat\sigma$ as a dimension to gain more insight into its influence.

We stress that from these figures, little if anything can be said about how ParaDiag scales when increasing $\Tidx$, the topic of \cref{sec:scale}. Instead, we conclude that \emph{for a fixed number of time steps}, ParaDiag can be expected to converge quickly unless both \begin{itemize}
    \item the equation in the absence of control evolves slowly \emph{relatively to the size of the time interval} ($\widehat\sigma \approx 0$); and
    \item there is little control \emph{relatively to the size of the time interval} ($\widehat\gamma \approx 0$).
\end{itemize}
If these conditions for potentially slow convergence are met, the difference between $\alpha$ values becomes important. \begin{itemize}
    \item When $\alpha=1$, the smaller $\widehat\sigma$ and $\widehat\gamma$ become, the more $\theta$ blows up. This will lead to preconditioned eigenvalues that lie far away from each other, resulting in slow convergence.
    \item When $\alpha=-1$, small values of $\widehat\sigma$ and $\widehat\gamma$ slightly pull the eigenvalues away from unity. However, they always stay relatively close due to \cref{cor:pd-track:anal:semidisc}. In addition, clustering may even become better for very small $\widehat\sigma$ and $\widehat\gamma$: \cref{fig:pd-track:plane:m1} shows that the worst clustering occurs at intermediate $\widehat\gamma$s.
\end{itemize}
This analysis explains our observations in \cref{fig:pd-track:eigs-1,fig:pd-track:eigs--1}. When $T=1$, $\widehat\sigma$ and $\widehat\gamma$ are large enough that the eigenvalues lie close to the edge of $\mathcal D_{0.5,+}$ for both $\alpha=\pm1$. When $T=10^{-4}$, $\widehat\sigma$ and $\widehat\gamma$ are very small and $\alpha=-1$ clusters, while $\alpha=1$ disperses.

\begin{figure}
    \centering
    \begin{subfigure}[b]{.4\textwidth}
        \includegraphics[width=\textwidth]{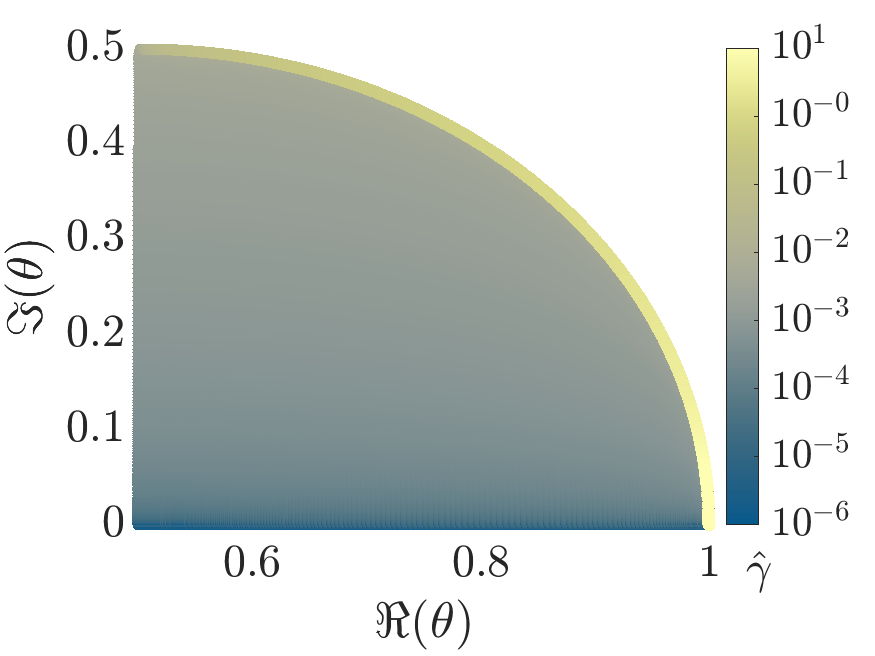}
        \caption{
            {$\theta$ when $\alpha=-1$}
        }
        \label{fig:pd-track:plane:m1}
    \end{subfigure}
    \hfill
    \begin{subfigure}[b]{.4\textwidth}
        \includegraphics[width=\textwidth]{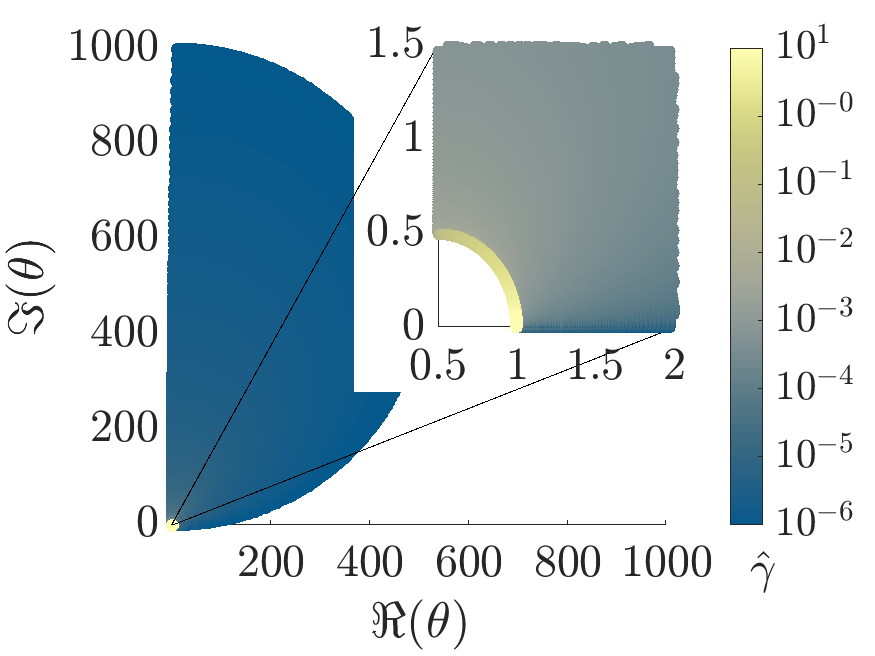}
        \caption{
            {$\theta$ when $\alpha=1$}
        }
        \label{fig:pd-track:plane:p1}
    \end{subfigure}
    \vfill
    \begin{subfigure}[b]{.4\textwidth}
        \includegraphics[width=\textwidth]{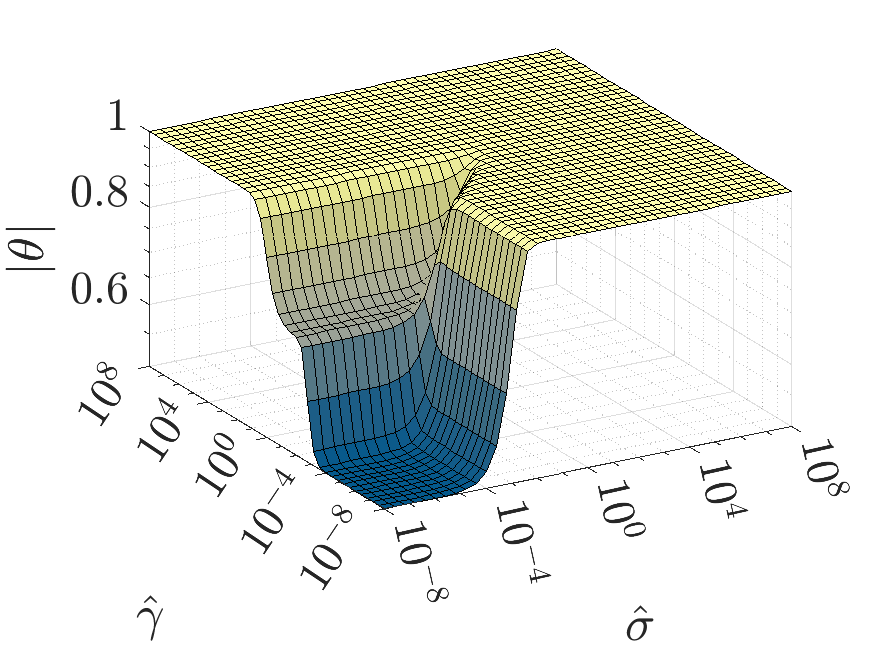}
        \caption{
            {$\abs\theta$ when $\alpha=-1$}
        }
        \label{fig:pd-track:eigs-abs:m1}
    \end{subfigure}
    \hfill
    \begin{subfigure}[b]{.4\textwidth}
        \includegraphics[width=\textwidth]{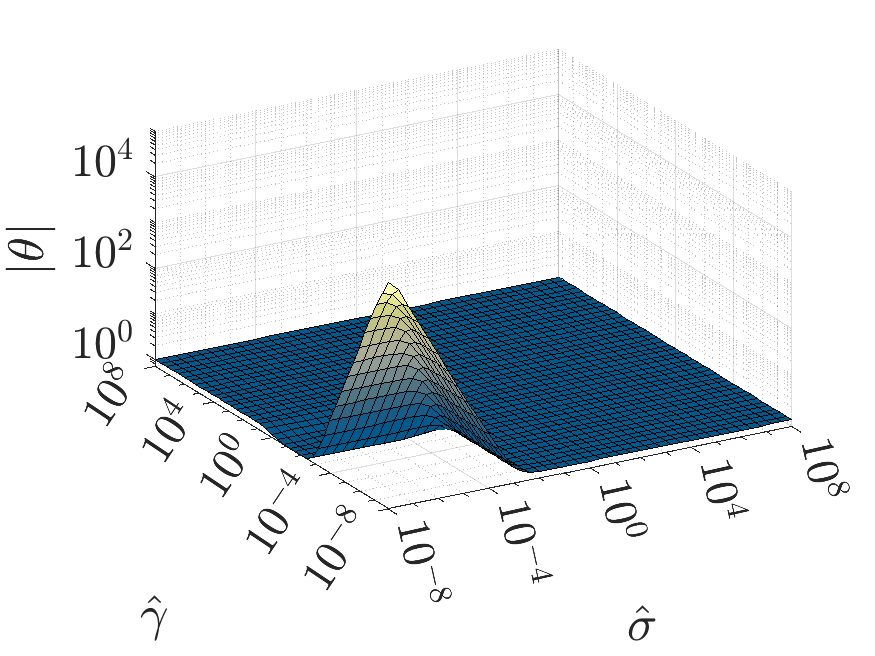}
        \caption{
            {$\abs\theta$ when $\alpha=1$}
        }
        \label{fig:pd-track:eigs-abs:p1}
    \end{subfigure}
    \caption{
        {Non-unity eigenvalue $\theta$ of $P(\alpha)^{-1}\widehat A$ with $\Im(\theta)\ge0$ for $\Tidx=1000$. Note the scale of the axes in the figures on the right, due to the effects of $\alpha=1$. The col\ou{}r maps used throughout this text were designed in \cite{crameriScientificColourMaps2021} to be col\ou{}r-vision--deficiency friendly.}
    }
    \label{fig:pd-track:eigs}
    \vspace{-.8cm}
\end{figure}

\newpage
Specifically for the \textsc{gmres} method, it is possible to harness \cref{cor:pd-track:anal:semidisc} into an upper bound on the convergence of the iterative method.
\begin{theorem}
    Consider the tracking-type all-at-once system with implicit-Euler time discreti\sz{}ation \cref{eq:pd-track:existing:aao-rescaled} where $\Tidx>4$ and $K$ is self-adjoint. When using ParaDiag with \textsc{gmres} preconditioned by $\kronm P(-1)$ (see \cref{eq:pd-track:alpha:Palpha}) and if $0<\varphi<1$ for all $\varphi$ (which occurs whenever $K$ is positive definite), the following holds. For any $0<\rho<2$, there exists a $\kappa_\rho > 0$ such that the residual $\kronv r^\iidx$ at \textsc{gmres} iteration $\iidx$ satisfies
    \begin{equation} \label{eq:thm:pd-track:anal:expgmres}
        \norm{\kronv r^\iidx}_2/\norm{\kronv r^0}_2 \le \kappa(\kronm V)\kappa_\rho\rho^{-\iidx}    \mper
    \end{equation}
    Here, $\kappa(\kronm V)$ is the condition number of the eigenvector matrix $\kronm V$ of $\kronm P(-1)^{-1}\kronm{\widehat A}$. The \textsc{gmres} residual decreases exponentially with a mesh- and problem-independent factor.
\end{theorem}
\begin{proof}
    From \cite{trefethenNumericalLinearAlgebra1997a}, we retrieve the formula
    \begin{equation} \label{eq:thm:pd-track:anal:expgmres:gmresorig}
        \norm{\kronv{r^\iidx}}_2/\norm{\kronv{r^0}}_2 \le \kappa(\kronm V)\inf_{p_\iidx\in\mathbb P_\iidx}{\sup_{\sigma\in\Sigma}{\abs{p_\iidx(\sigma)}}}
    \end{equation}
    which asks us to solve a polynomial-approximation problem: find a degree-$\iidx$ polynomial that takes the value $1$ at the origin and, yet, is as small as possible on all the eigenvalues of the preconditioned matrix. However, if we want a generally applicable bound, we do not know these eigenvalues. Luckily, we have \cref{cor:pd-track:anal:semidisc}: if we can find a polynomial that is small on the \emph{entire} semi-dis\ukus{c}{k} $\mathcal D_{0.5,+}$, then a fortiori, it must also be small on whatever eigenvalues a specific problem happens to generate.

    To eliminate the explicit condition $p(0)=1$, we take the following steps. If we can find a degree-$k$ polynomial $\widehat p$ that satisfies $\widehat p(0)=0$ and approximates $1$ on $\mathcal D_{0.5,+}$, it has the same error as $p(z) = 1-\widehat p(z)$. Since $\widehat p(0)=0$, it must be that $\widehat p(z) = zq(z)$, where $q$ is of degree $\iidx-1$. If we find a polynomial $q$ that approximates $1/z$ on $\mathcal D_{0.5,+}$, we have a function $\widehat p(z) = zq(z)$ which approximates $1$ with no higher error than that of $q$ in approximating $1/z$ (this follows easily from the fact that $\abs z\le1$).

    In summary, \cref{eq:thm:pd-track:anal:expgmres:gmresorig}'s rightmost factor is bounded by the best degree-$(\iidx-1)$ polynomial-approximation error to $1/z$ on $\mathcal D_{0.5,+}$, which is bounded by \cref{lmm:apdx-proof:pd-track:polappr}.
\end{proof}

    \subsection{Generali\sz{}ing past self-adjoint problems} \label{sec:pd-track:gen}
    We now extend ParaDiag to the more general, non-self-adjoint setting where $K\ne K^*$ is possible, resulting in the optimality system \cref{eq:intro:intro:optsys}. The method supports using any adjoint spatial discreti\sz{}ation, which could be of interest, but we will limit ourselves to $K^*$. Then \cref{eq:pd-track:existing:aao} becomes
\begin{equation} \label{eq:pd-track:gen:aao}
    \kronm A\begin{litmat}
        \kronv y\\ \kronv{\ad}\\
    \end{litmat} \coloneqq \left(\begin{litmat}
        B & \tau\frac{I_t}{\gamma}\\
        -\tau I_t & B^\trsp\\
    \end{litmat} \kron I_x + \tau\begin{litmat}
        I_t \kron K\\&I_t \kron K^*\\
    \end{litmat} \right)\begin{litmat}
        \kronv y\\ \kronv{\ad}\\
    \end{litmat} = \begin{litmat}
        \kronv{b_1}\\ \kronv{b_2}\\
    \end{litmat}
\end{equation}
and, after rescaling,
\begin{equation} \label{eq:pd-track:gen:aao-rescaled}
    \kronm{\widehat A}\begin{litmat}
        \kronv y\\ \kronv{\widehat\ad}\\
    \end{litmat} \coloneqq \left(\begin{litmat}
        B & \tau\frac{I_t}{\sqrt\gamma}\\
        -\tau\frac{I_t}{\sqrt\gamma} & B^\trsp\\
    \end{litmat} \kron I_x + \tau\begin{litmat}
        I_t \kron K\\&I_t \kron K^*\\
    \end{litmat} \right)\begin{litmat}
        \kronv y\\ \kronv{\widehat\ad}\\
    \end{litmat} = \begin{litmat}
        \kronv{b_1}\\ \kronv{\widehat b_2}\\
    \end{litmat}    \mper
\end{equation}
We suggest an alpha-circulant preconditioner that replaces $B$ by $C(\alpha)$, factori\sz{}ing as
\begin{equation} \label{eq:pd-track:gen:P}
    \kronm P(\alpha) \coloneqq (V\kron I_x)\Biggl(\begin{litmat}
        D(\alpha) & \frac{\tau I_t}{\sqrt\gamma}\\
        -\frac{\tau I_t}{\sqrt\gamma} & D(\alpha)^*\\
    \end{litmat} \kron I_x + \tau\begin{litmat}
        I_t \kron K\\&I_t \kron K^*\\
    \end{litmat}\Biggr)(V^{-1}\kron I_x)
\end{equation}
by using the property \cref{eq:pd-track:alpha:Calpha-fact}. As $K\ne K^*$ is possible in this generali\sz{}ed case, a further factori\sz{}ation such as the one from \cref{eq:pd-track:existing:P-fac} to \cref{eq:pd-track:existing:P-fac-further} cannot be reproduced with this preconditioner. However, inversion of \cref{eq:pd-track:gen:P}'s middle factor can already be paralleli\sz{}ed in the time direction; all time steps have been decoupled. This procedure only misses out on the additional paralleli\sz{}ation factor of 2 that decoupling the state and adjoint equations in \cref{eq:pd-track:existing:P-fac-further} provides in the self-adjoint case.

\Cref{alg:pd-track:new} incorporates the alpha-circulant improvement from \cref{sec:pd-track:alpha}, as well as the above generali\sz{}ation. It can be compared to \cref{alg:pd-track:existing}.

\begin{algorithm}
    \caption{ParaDiag for solving the generali\sz{}ed tracking problem \cref{eq:pd-track:gen:aao}} \label{alg:pd-track:new}
    \begin{tabular}{rl}
        \textbf{Input:}  &Vectors $\kronv{b_1}$ and $\kronv{b_2}$ defined by \cref{eq:pd-track:existing:Bbb}\\
                            &\emph{Arbitrary} matrix $K$ characterising the problem by \cref{eq:intro:intro:optsys}\\
                            &Matrix $D(\alpha)$ following from the time discreti\sz{}ation by \cref{eq:pd-track:alpha:Calpha-fact} ($\hspace{.03cm}\abs\alpha=1$)\\
        \textbf{Output:} &The vectors $\kronv y$ and $\kronv\ad=\sqrt\gamma\kronv{\widehat\ad}$ that solve \cref{eq:pd-track:gen:aao}\\
    \end{tabular}
    \begin{algorithmic}[1]
        \State Rescale $\kronv{\widehat b_2} =  \kronv{b_2}/\sqrt\gamma$.
        \State Solve \cref{eq:pd-track:gen:aao-rescaled} for $\kronv y$ and $\kronv{\widehat\ad}$ using an iterative method, with preconditioner $\kronm P(\alpha)$ from \cref{eq:pd-track:gen:P}. When asked to compute $\bigl[\begin{smallmatrix}\kronv x\\ \kronv z\end{smallmatrix}\bigr] = \kronm P(\alpha)^{-1}\bigl[\begin{smallmatrix}\kronv v\\ \kronv w\end{smallmatrix}\bigr]$:
            \State \hskip1cm Calculate $\kronv{r_1} \coloneqq (\dfm \Gamma_\alpha \kron I_x)\kronv{v}$, $\kronv{s_1} \coloneqq (\dfm \Gamma_\alpha \kron I_x)\kronv{w}$ with the (parallel) \textsc{fft}.
            \State \hskip1cm For $\tidx=\{1, \ldots, \widehat\Tidx\}$, solve (in parallel)
                \begin{equation}
                    \begin{litmat}\mv{r_{2,\tidx}}\\\mv{s_{2,\tidx}}\end{litmat}\coloneqq
                    \begin{litmat}
                        d_{\tidx,\tidx}(\alpha)I_x+\tau K & \frac\tau{\sqrt\gamma}I_x\\
                        -\frac\tau{\sqrt\gamma}I_x & d_{\tidx,\tidx}(\alpha)^*I_x + \tau K^*\\
                    \end{litmat}^{-1}
                    \begin{litmat}\mv{r_{1,\tidx}}\\\mv{s_{1,\tidx}}\end{litmat}    \mper
                \end{equation}
            \State \hskip1cm Calculate $\kronv{x} = (\Gamma_\alpha^{-1}\dfm ^* \kron I_x)\kronv{r_2}$, $\kronv{z} = (\Gamma_\alpha^{-1}\dfm ^* \kron I_x)\kronv{s_2}$ with the (parallel) \textsc{fft}.
    \end{algorithmic}
\end{algorithm}

\section{ParaDiag for terminal-cost objectives} \label{sec:pd-tc}
    Both the literature on ParaDiag and this paper have thus far focused on the tracking objective in \cref{eq:intro:intro:obj}. We next develop a ParaDiag-type preconditioner for problems with the terminal-cost objective function, without requiring self-adjointness. The method is designed in \cref{sec:pd-tc:new}, after which it is analy\sz{}ed for self-adjoint problems in \cref{sec:pd-tc:anal,sec:pd-tc:interp}.

    \subsection{A new preconditioner} \label{sec:pd-tc:new}
    The optimality system in the terminal-cost case can be discreti\sz{}ed with time step $\tau$ to form the all-at-once system
\begin{equation} \label{eq:pd-tc:new:aao}
    \kronm A\begin{litmat}
        \kronv y\\ \kronv\ad\\
    \end{litmat} \coloneqq \left(\begin{litmat}
        B & \frac\tau\gamma I_t\\
        -E & B^\trsp\\
    \end{litmat} \kron I_x + \tau\begin{litmat}
        I_t \kron K\\-E \kron K^*&I_t \kron K^*\\
    \end{litmat} \right)\begin{litmat}
        \kronv y\\ \kronv\ad\\
    \end{litmat} = \kronv b
\end{equation}
where $E$ is a matrix with as only non-zero a one in the bottom right corner. Recall that the exposition assumes an explicit Euler discreti\sz{}ation, which implies
\begin{equation} \label{eq:pd-tc:new:Bb}
    B = \left[\begin{smallmatrix}
        1 &\\-1 & 1 &\\
        & \ddots & \ddots\\
        && -1 & 1
    \end{smallmatrix}\right] \quad \text{and} \quad \kronv b = \begin{litmat}\mv{y^\trsp_\mathrm{init}} & 0 & \ldots & 0 & -((I_x+\tau K^*)\mv{y_\mathrm{target}})^\trsp\end{litmat}^\trsp    \mper
\end{equation}
In contrast to the tracking situation, the discreti\sz{}ation point at time $t=T$ cannot be eliminated due to the more complex terminal condition in \cref{eq:intro:intro:optsys}. Thus $B$ is $\Tidx\times\Tidx$.

ParaDiag methods are fully reliant on the presence of good preconditioners, preferably with a mesh-independent convergence rate. Such a preconditioner must be invertible efficiently and in parallel. Leaving the bottom-left block of \cref{eq:pd-tc:new:aao} out of the preconditioner makes this task significantly easier. Indeed, it allows replacing the $B$ blocks by alpha-circulant $C(\alpha)$ blocks to form the preconditioner
\begin{equation} \label{eq:pd-tc:new:Palpha}
    \kronm P(\alpha) = \begin{litmat}
        C(\alpha) & \frac\tau\gamma I_t\\
        &C(\alpha)^*\\
    \end{litmat} \kron I_x + \tau\begin{litmat}
        I_t\kron K\\&I_t\kron K^*\\
    \end{litmat}    \mcom
\end{equation}
which is block-triangular. Thus multiplication by $\kronm P(\alpha)^{-1}$ is possible by first inverting the bottom-right block of \cref{eq:pd-tc:new:Palpha} (which pertains to the adjoint variable $\lambda$) and only then solving a second system to find the state $y$. Due to this procedure, $C(\alpha)$ and $C(\alpha)^*$ no longer need to be simultaneously diagonali\sz{}able, and $\abs\alpha$ can be smaller than $1$, in contrast to the tracking method. \Cref{alg:pd-tc:new} spells out how to solve \cref{eq:pd-tc:new:aao} using the ParaDiag method this subsection proposes.

\begin{algorithm}
    \caption{ParaDiag procedure for solving the terminal-cost problem \cref{eq:pd-tc:new:aao}} \label{alg:pd-tc:new}
    \begin{tabular}{rl}
        \textbf{Input:}  &Vector $\kronv b$ defined by \cref{eq:pd-tc:new:Bb}\\
                            &\emph{Arbitrary} matrix $K$ characterising the problem by \cref{eq:intro:intro:optsys}\\
                            &Matrix $D(\alpha)$ following from the time discreti\sz{}ation by \cref{eq:pd-track:alpha:Calpha-fact} ($\alpha\ne0$)\\
        \textbf{Output:} &The vectors $\kronv y$ and $\kronv\ad$ that solve \cref{eq:pd-tc:new:aao}\\
    \end{tabular}
    \begin{algorithmic}[1]
        \State Solve \cref{eq:pd-tc:new:aao} for $\kronv y$ and $\kronv\ad$ using an iterative method, with preconditioner $\kronm P(\alpha)$ from \cref{eq:pd-tc:new:Palpha}. When asked to compute $\bigl[\begin{smallmatrix}\kronv x\\ \kronv z\end{smallmatrix}\bigr] = \kronm P(\alpha)^{-1}\bigl[\begin{smallmatrix}\kronv v\\ \kronv w\end{smallmatrix}\bigr]$:
            \LineCommentWithSkip Phase 1: invert the bottom-right block
            \State \hskip 1cm Calculate $((\mv{s_{1,1}})^\trsp, \ldots, (\mv{s_{1,\Tidx}})^\trsp)^\trsp \coloneqq (\dfm \Gamma_\alpha^{-*} \kron I_x)\kronv{w}$ with the (parallel) \textsc{fft}.
            \State \hskip 1cm For $\tidx=\{1, \ldots, \Tidx\}$, solve (in parallel)
            \begin{equation}
                \mv{s_{2,\tidx}} \coloneqq (d_{\tidx,\tidx}(\alpha)^*I_x + \tau K^*)^{-1}\mv{s_{1,\tidx}}
            \end{equation}
            \hskip 1cmand assemble $\kronv{s_2} \coloneqq ((\mv{s_{2,1}})^\trsp, \ldots, (\mv{s_{2,\Tidx}})^\trsp)^\trsp$.
            \State \hskip 1cm Calculate $\kronv{z} = (\Gamma_\alpha^*\dfm ^* \kron I_x)\kronv{s_2}$ with the (parallel) \textsc{fft}.
            \vspace{.3cm}
            \LineCommentWithSkip Phase 2: invert the rest of the matrix
            \State \hskip 1cm Set $\kronv{r_1} = \kronv v - \frac\tau\gamma\kronv z$.
            \State \hskip 1cm Calculate $((\mv{r_{2,1}})^\trsp, \ldots, (\mv{r_{2,\Tidx}})^\trsp)^\trsp \coloneqq (\dfm \Gamma_\alpha \kron I_x)\kronv{r_1}$ with the (parallel) \textsc{fft}.
            \State \hskip 1cm For $\tidx=\{1, \ldots, \Tidx\}$, solve (in parallel)
            \begin{equation}
                \mv{r_{3,\tidx}} \coloneqq (d_{\tidx,\tidx}(\alpha)I_x + \tau K)^{-1}\mv{r_{2,\tidx}}
            \end{equation}
            \hskip 1cmand assemble $\kronv{r_3} \coloneqq ((\mv{r_{3,1}})^\trsp, \ldots, (\mv{r_{3,\Tidx}})^\trsp)^\trsp$.
            \State \hskip 1cm Calculate $\kronv{x} = (\Gamma_\alpha^{-1}\dfm ^* \kron I_x)\kronv{r_3}$ with the (parallel) \textsc{fft}.
    \end{algorithmic}
\end{algorithm}

    \subsection{Analytic eigenvalue expressions} \label{sec:pd-tc:anal}
    As was the case for tracking, we will formulate analytic eigenvalue results for the special case of a self-adjoint matrix $K=K^*$. The preparatory steps from \cref{sec:pd-track:anal} are straightforward to repeat: we perform the same rescaling, resulting in
\begin{equation} \label{eq:pd-tc:anal:resc}
    \kronm{A_\mathrm p} \coloneqq \left[\begin{smallarray}{cccc|cccc}
        I_x & & & & \Psi\\
        -\Phi & I_x & & & & \Psi\\
        & \ddots & \ddots & & & & \ddots\\
        & & -\Phi & I_x & & & & \Psi\\
        \cmidrule(lr){1-4}\cmidrule(lr){5-8}
         & & & & I_x & -\Phi\\
        & & & & & \ddots & \ddots\\
        & & & & & & I_x & -\Phi\\
        & & & -I_x & & & & I_x\\
    \end{smallarray}\right], \kronm P_\mathrm p(\alpha) \coloneqq\left[\begin{smallarray}{cccc|cccc}
        I_x & & & -\alpha\Phi & \Psi\\
        -\Phi & I_x & & & & \Psi\\
        & \ddots & \ddots & & & & \ddots\\
        & & -\Phi & I_x & & & & \Psi\\
        \cmidrule(lr){1-4}\cmidrule(lr){5-8}
         & & & & I_x & -\Phi\\
        & & & & & \ddots & \ddots\\
        & & & & & & I_x & -\Phi\\
        & & & & -\alpha\Phi & & & I_x\\
    \end{smallarray}\right]
\end{equation}
where $\alpha\in\mathbb R$ was assumed. This time, $\Phi=(I_x+\tau K)^{-1}$ and $\Psi=\frac\tau\gamma(I_x+\tau K)^{-1}$. We again perform a decomposition to the scalar case, such that the eigenvalues of $\kronm P(\alpha)^{-1}\kronm A = \kronm P_\mathrm p(\alpha)^{-1}\kronm A_\mathrm p$ are the union of those of
\begin{equation}
    P_\sigma(\alpha)^{-1}A_\sigma = \left[\begin{smallarray}{cccc|cccc}
        1 & & & -\alpha\varphi & \psi\\
        -\varphi & 1 & & & & \psi\\
        & \ddots & \ddots & & & & \ddots\\
        & & -\varphi & 1 & & & & \psi\\
        \cmidrule(lr){1-4}\cmidrule(lr){5-8}
        & & & & 1 & -\varphi\\
        & & & & & \ddots & \ddots\\
        & & & & & & 1 & -\varphi\\
        & & &  & -\alpha\varphi & & & 1\\
    \end{smallarray}\right]^{-1}\left[\begin{smallarray}{cccc|cccc}
        1 & & & & \psi\\
        -\varphi & 1 & & & & \psi\\
        & \ddots & \ddots & & & & \ddots\\
        & & -\varphi & 1 & & & & \psi\\
        \cmidrule(lr){1-4}\cmidrule(lr){5-8}
        & & & & 1 & -\varphi\\
        & & & & & \ddots & \ddots\\
        & & & & & & 1 & -\varphi\\
        & & & -1 & & & & 1\\
    \end{smallarray}\right]    \mper
\end{equation}
for all eigenvalues $\sigma$ of $K$, where $\varphi = (1+\tau\sigma)^{-1}$ and $\psi=\frac\tau\gamma(1+\tau\sigma)^{-1}$. We can eliminate $\tau$ by defining $\widehat\sigma\coloneqq\tau\sigma$ and $\widehat\gamma\coloneqq\frac\tau\gamma$, leading to
\begin{equation} \label{eq:pd-tc:anal:phipsi}
    \varphi = (1 + \widehat\sigma)^{-1} \quad \text{and} \quad \psi = \widehat\gamma(1+\widehat\sigma)^{-1}    \mper
\end{equation}
These equations are identical to \cref{eq:pd-track:anal:phipsi}, but notice that the definition of $\widehat\gamma$ is different.

The problem has been reduced to finding the eigenvalues $\theta$ of
\begin{equation}
    P_\sigma(\alpha)^{-1}A_\sigma = I_t+P_\sigma(\alpha)^{-1}(A_\sigma-P_\sigma(\alpha)) \eqqcolon I_t + P_\sigma(\alpha)^{-1}R_\sigma    \mper
\end{equation}
We get $\theta = 1 + \omega$, where the $\omega$s are eigenvalues of $P_\sigma(\alpha)^{-1}R_\sigma$, studied in \cref{thm:pd-tc:eigs}.

\begin{theorem} \label{thm:pd-tc:eigs}
    Let $\Tidx > 3$ and $\alpha,\varphi,\psi\in\mathbb R$ with $\alpha\ne0$. The $2\Tidx\times2\Tidx$ matrix
    \begin{equation} \label{thm:pd-tc:eigs:M}
        M  = \left[\begin{smallarray}{cccc|cccc}
            1&&&-\alpha\varphi&\psi\\
            -\varphi & 1 &&&&\psi\\
            &\ddots&\ddots&&&&\ddots\\
            &&-\varphi&1&&&&\psi\\
            \cmidrule(lr){1-4}\cmidrule(lr){5-8}&&&&1&-\varphi\\
            &&&&&\ddots&\ddots\\
            &&&&&&1&-\varphi\\
            &&&&-\alpha\varphi&&&1\\
        \end{smallarray}\right]^{-1}\left[\begin{smallarray}{cccc|cccc}
            \phantom{1}&&&\phantom{-}\alpha\varphi&\phantom{\psi}\\
            \phantom{-\varphi} & \phantom{1}&&&&\phantom{\psi}\\
            &\phantom{\ddots}&\phantom{\ddots}&&&&\phantom{\ddots}\\
            &&\phantom{-\varphi}&\phantom{1}&&&&\phantom{\psi}\\
            \cmidrule(lr){1-4}\cmidrule(lr){5-8}
            &&&&\phantom{1}&\phantom{-\varphi}\\
            &&&&&\phantom{\ddots}&\phantom{\ddots}\\
            &&&&&&\phantom{1}&\phantom{-\varphi}\\
            &&&-1&\phantom{-}\alpha\varphi&&&\phantom{1}\\
        \end{smallarray}\right]
    \end{equation}
    where $\varphi\ne\pm1$ has two potentially non-zero eigenvalues $\omega_{\{1,2\}}$: those of the matrix
    \begin{equation} \label{eq:lmm:pd-tc:anal:eigs:Mred}
        M_\mathrm{red} = \begin{litmat}
            \frac{\alpha\varphi^\Tidx}{1-\alpha\varphi^\Tidx}+\frac\psi{(1-\alpha\varphi^\Tidx)^2}\frac{1-\varphi^{2\Tidx}}{1-\varphi^2} & -\frac{\alpha\varphi\psi}{(1-\alpha\varphi^\Tidx)^2}\frac{1-\varphi^{2\Tidx}}{1-\varphi^2}\\
            -\frac{\varphi^{\Tidx-1}}{1-\alpha\varphi^\Tidx} & \frac{\alpha\varphi^\Tidx}{1-\alpha\varphi^\Tidx}\\
        \end{litmat}    \mper
    \end{equation}
    When $\alpha$ goes to zero, this simplifies to
    \begin{equation} \label{eq:lmm:pd-tc:anal:eigs:lim}
        \lim_{\alpha\rightarrow0}\omega_1 = \psi\frac{1-\varphi^{2\Tidx}}{1-\varphi^2} \quad \text{and} \quad \lim_{\alpha\rightarrow0}\omega_2 = 0    \mper
    \end{equation}
\end{theorem}
\begin{proof}
    The proof is given in \cref{sec:apdx-pdtce}.
\end{proof}

\begin{corollary}
    Consider a terminal-cost all-at-once system with implicit-Euler time discreti\sz{}ation \cref{eq:pd-tc:new:aao} where $\Tidx>3$ and $K$ is self-adjoint with eigenvalues $\left\{\sigma_\sidx\right\}_{\sidx=1}^{\Sidx}$. When using ParaDiag with preconditioner $\kronm P(\alpha)$ (see \cref{eq:pd-tc:new:Palpha}), the eigenvalues of the preconditioned system matrix are all either unity or equal to
    \begin{equation}
        \theta_{\sidx,\{1,2\}} = 1 + \omega_{\sidx,\{1,2\}}
    \end{equation}
    where $\omega_{\sidx,\{1,2\}}$ are given as the eigenvalues of \cref{eq:lmm:pd-tc:anal:eigs:Mred}, having filled in $\varphi = (1+\tau\sigma_\sidx)^{-1}$ and $\psi = \frac\tau\gamma(1+\tau\sigma_\sidx)^{-1}$ (on the condition that $\varphi\ne\pm1$).
\end{corollary}

    \subsection{Interpreting the eigenvalue results} \label{sec:pd-tc:interp}
    In practice, $\alpha$ in \cref{eq:pd-tc:new:Palpha} can usually be taken small enough such that the limit \cref{eq:lmm:pd-tc:anal:eigs:lim} is valid (the only constraint is rounding errors occurring for very small $\alpha$ \cite{ganderDirectTimeParallel2019a,wuParallelCoarseGrid2018a}). Given a rescaled eigenvalue $\widehat\sigma$ of $K$, \cref{fig:pd-tc:eigs} plots the non-unity preconditioned eigenvalue $\theta$ in this limit.
\begin{figure}
    \centering
    \includegraphics[width=.4\textwidth]{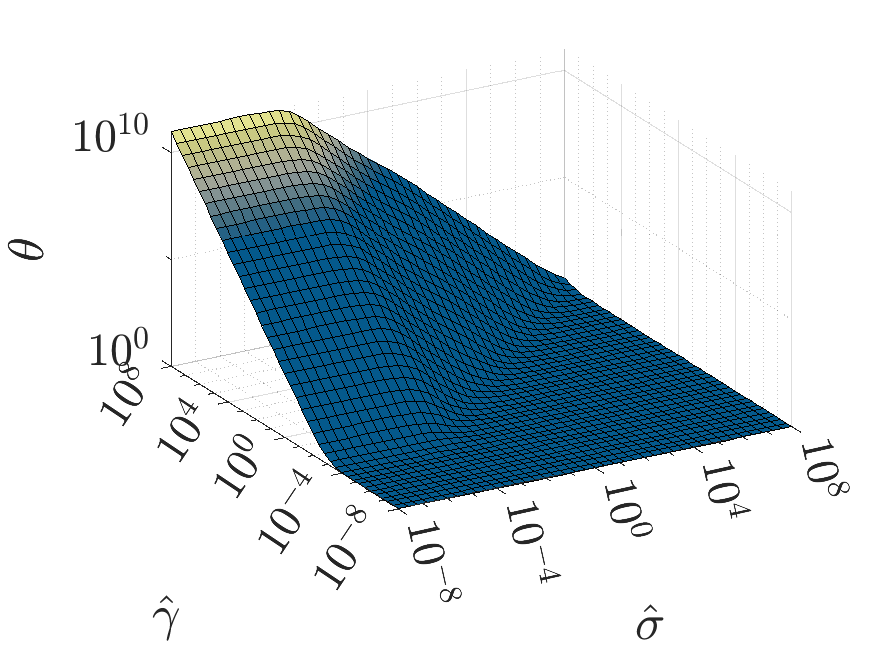}
    \caption{
        \vspace{-1cm}
        The non-unity preconditioned eigenvalues $\theta$ of $P(0)^{-1}A$ with $\Tidx=1000$
        \vspace{.5cm}
    }
    \label{fig:pd-tc:eigs}
\end{figure}
Based on this figure, \emph{for a fixed number of time steps}, poor convergence is expected when both \begin{itemize}
    \item the equation in the absence of control evolves slowly \emph{relatively to the size of the time interval} ($\widehat\sigma\approx0$); and
    \item there is a lot of control \emph{relatively to the size of the time interval} ($\widehat\gamma\gg0$).
\end{itemize}
The first condition is the same as for the tracking preconditioner, but the second is different. Interestingly, the ParaOpt algorithm \cite{ganderPARAOPTPararealAlgorithm2020a}, which also treats terminal-cost objectives, struggles in the high-$\widehat\gamma$ regime as well.

\section{Parallel-scaling analysis for self-adjoint problems} \label{sec:scale}
    An oft-used metric in the context of parallel algorithms is \emph{weak scalability} (for time-parallel methods, it was studied in e.g.\ \cite{benedusiExperimentalComparisonSpacetime2021,caceressilvaParallelintimePararealImplementation2014}). The aim is that a program's execution time stays constant when increasing the problem size (in our case the number of time steps $L$, as we investigate \emph{time}-parallelism) in tandem with the number of processors, keeping their ratio constant. For our optimal-control problem \cref{eq:intro:intro:optprob}, we identify two regimes \cite{ganderPARAOPTPararealAlgorithm2020a}. \begin{itemize}
    \item If we increase the time horizon $T$ together with $L$, the time step $\tau$ stays constant. In this regime, $\widehat\sigma$ and $\widehat\gamma$ (from \cref{eq:pd-track:anal:phipsi} or \cref{eq:pd-tc:anal:phipsi}, depending on the objective function) do not change. The amount of work is increased by an expanding time scope, not by using a more accurate discreti\sz{}ation.
    \item We can also keep $T$ constant but instead increase the amount of time steps $L$ by lowering $\tau$. Then $\widehat\sigma$ and $\widehat\gamma$ increase with it. The amount of work is increased by using a more fine-grained mesh for the same problem.
\end{itemize}
This section will use the analytic results from \cref{sec:pd-track:anal,sec:pd-tc:anal} to perform a theoretical analysis of ParaDiag's weak scaling, which \cref{sec:num} later verifies in practice. The approach is to assume the inversion of our preconditioners scales well in all regimes, as attested to by previous ParaDiag algorithms that use similar preconditioners \cite{goddardNoteParallelPreconditioning2019,wuDiagonalizationbasedParallelintimeAlgorithms2020b,wuParallelInTimeBlockCirculantPreconditioner2020a}. Then, all that needs to be analy\sz{}ed is the number of such inversions: if the iterative solver's iteration count stays constant when increasing the problem size, we have achieved good weak scalability. As a proxy for the actual iteration count, we will use the distribution of the preconditioned eigenvalues -- if they converge when increasing time parallelism, we will assume for the iteration count to do the same.

This section is limited to implicit Euler and self-adjoint, dissipative equations. We thus have $\sigma>0$ and the obvious $\gamma, T > 0$. We aim to show that each of the eigenvalues $\theta$ converges to some finite, non-zero value in the relevant scaling limit.

    \subsection{Increasing the time horizon}
    \begin{figure}
    \centering
    \begin{subfigure}[b]{.3\textwidth}
        \includegraphics[width=\textwidth]{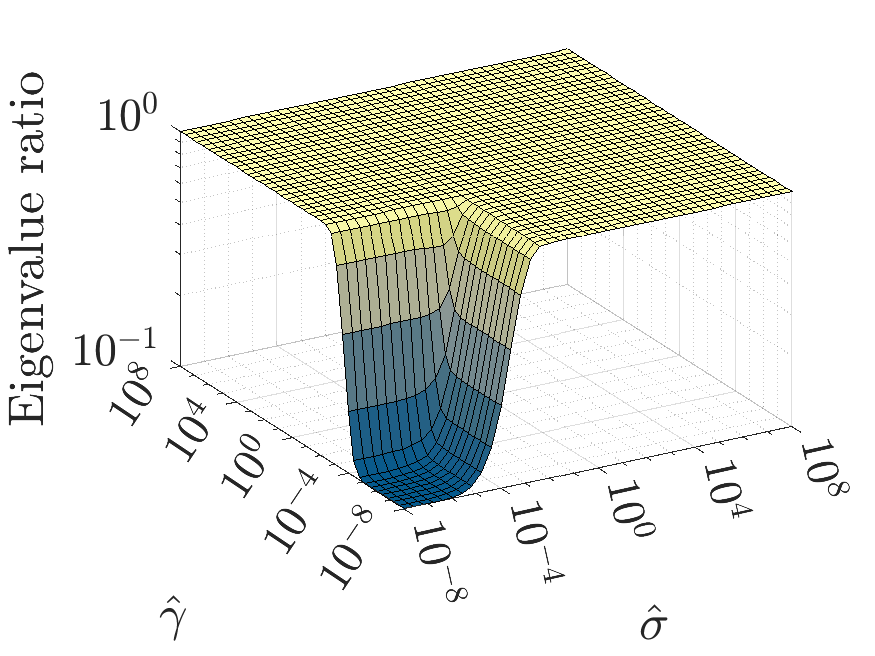}
        \caption{
            Tracking ($\alpha=1$)
            \vspace{-.4cm}
        }
        \label{fig:scale:hor:tr-pT}
    \end{subfigure}
    \hfill
    \begin{subfigure}[b]{.3\textwidth}
        \includegraphics[width=\textwidth]{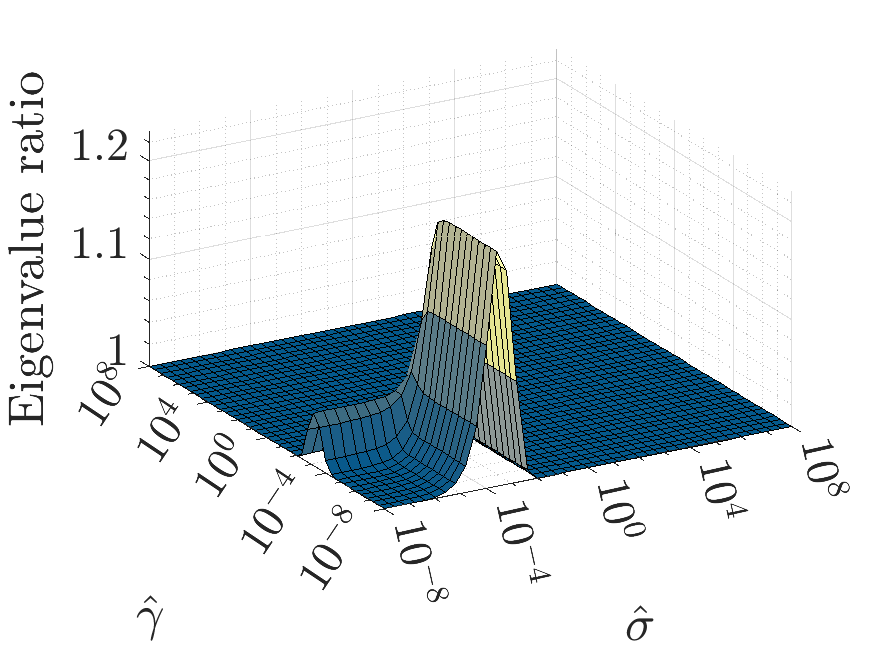}
        \caption{
            Tracking ($\alpha=-1$)
            \vspace{-.4cm}
        }
        \label{fig:scale:hor:tr-mT}
    \end{subfigure}
    \hfill
    \begin{subfigure}[b]{.3\textwidth}
        \includegraphics[width=\textwidth]{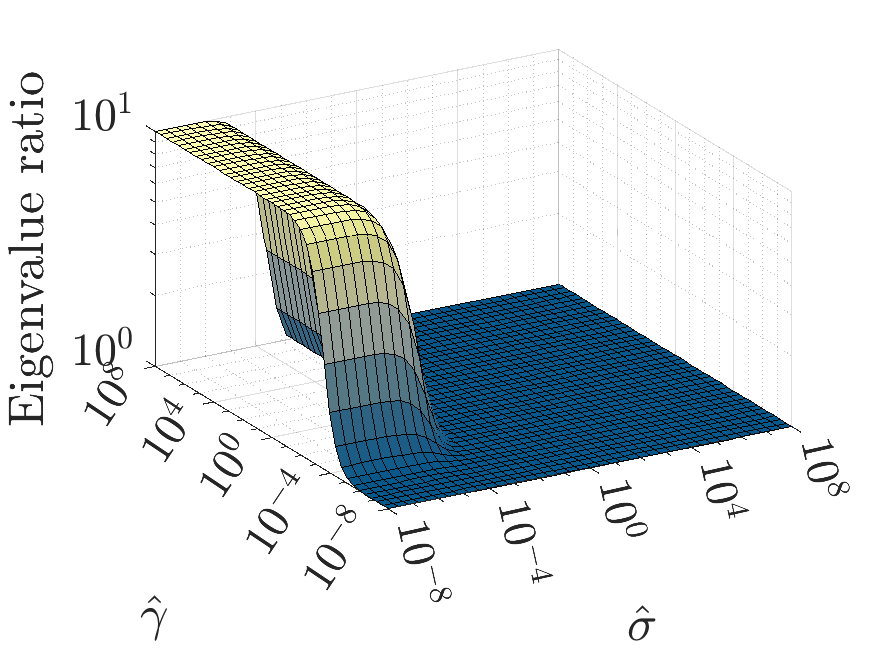}
        \caption{
            Terminal cost ($\alpha\rightarrow0$)
            \vspace{-.4cm}
        }
        \label{fig:scale:hor:tc-T}
    \end{subfigure}
    \caption{
        Ratio $\abs{\theta(\Tidx=10^4)}/\abs{\theta(\Tidx=10^3)}$ of the preconditioned-eigenvalue magnitudes when scaling $\Tidx$ from $10^3$ to $10^4$ through $T$, for different preconditioners
    }
    \label{fig:scale:hor}
    \vspace{-.4cm}
\end{figure}

Increasing $T$ while keeping $\tau$ constant does not affect $\widehat\sigma$ or $\widehat\gamma$. As a result, the only change in \cref{eq:thm:pd-track:eigs:eigs,eq:lmm:pd-tc:anal:eigs:lim} is that of $\widehat\Tidx = \Tidx-1$.

\paragraph*{Tracking}
In \cref{eq:thm:pd-track:eigs:eigs}, we have $0 < z_2 < 1 < z_1$ (see \cref{lmm:apdx-proof:pd-track:zs}(a)). As a result, in the limit for large $T$, $z_1^{\widehat\Tidx} \rightarrow \infty$ and $z_2^{\widehat\Tidx} \rightarrow 0$. That means that, for both $\alpha=\pm1$, the eigenvalues of the preconditioned matrix converge to
\begin{equation}
    \lim_{\Tidx\rightarrow\infty,T\rightarrow\infty,T=\tau\Tidx}\theta_{\{1,2\}} = 1 + \frac1{z_2-z_1}(-(z_2-\varphi\pm\psi\iu)) = \frac{z_1-\varphi\pm\psi\iu}{z_1-z_2}    \mper
\end{equation}
This is finite and non-zero; as assumed in the intro to \cref{sec:scale}, weak scalability can be expected. \Cref{fig:scale:hor:tr-pT,fig:scale:hor:tr-mT} start from finite $\Tidx=10^3$ and show that $\abs\theta$ does not increase significantly when scaling $\Tidx$ to $10^4$. It even decreases when $\widehat\sigma,\widehat\gamma\gtrsim 0$, where \cref{fig:pd-track:plane:p1} shows that $\abs\theta$ is high to start with.\clearpage

\paragraph*{Terminal cost}
Something very similar occurs in \cref{eq:lmm:pd-tc:anal:eigs:lim}. From $0<\varphi<1$, it follows that $\varphi^{2\Tidx}\rightarrow0$ and the non-zero eigenvalue $\theta_1$ approaches 
\begin{equation}
    \lim_{\Tidx\rightarrow\infty,T\rightarrow\infty,T=\tau\Tidx,\alpha\rightarrow0}\theta_1 = \psi/(1-\varphi^2)    \mcom
\end{equation}
which is finite and non-zero since $\sigma,\gamma,T>0$. In the limit $\Tidx\rightarrow\infty$, weak scalability is expected. \Cref{fig:scale:hor:tc-T} confirms that $\theta$ scales well even when $\Tidx$ is finite, except for very low $\widehat\sigma$ values, where the asymptotic region is not yet reached.

    \subsection{Decreasing the time step}
    \begin{figure}
    \centering
    \begin{subfigure}[b]{.3\textwidth}
        \includegraphics[width=\textwidth]{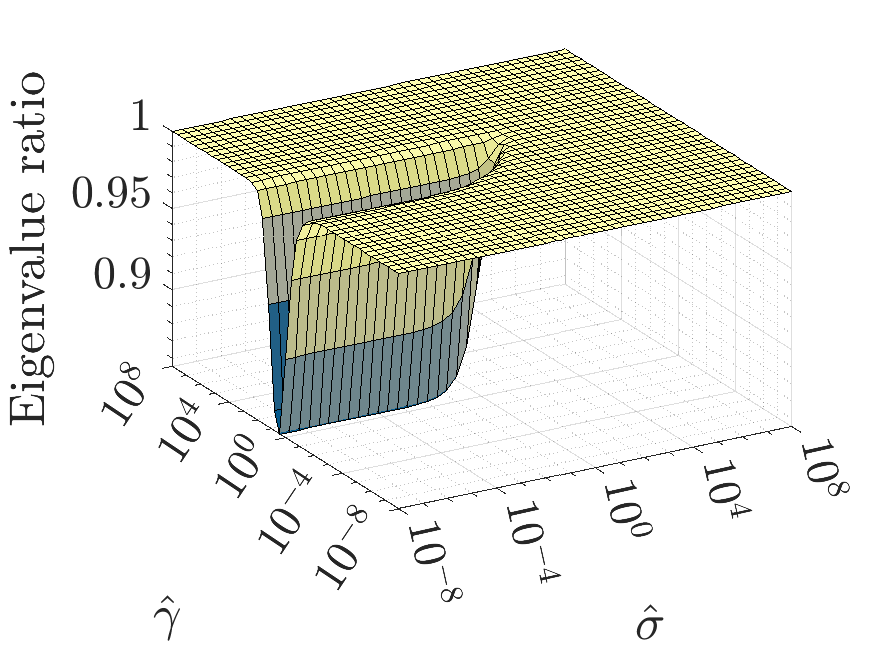}
        \caption{
            Tracking ($\alpha=1$)
            \vspace{-.3cm}
        }
        \label{fig:scale:tau:tr-pDT}
    \end{subfigure}
    \hfill
    \begin{subfigure}[b]{.3\textwidth}
        \includegraphics[width=\textwidth]{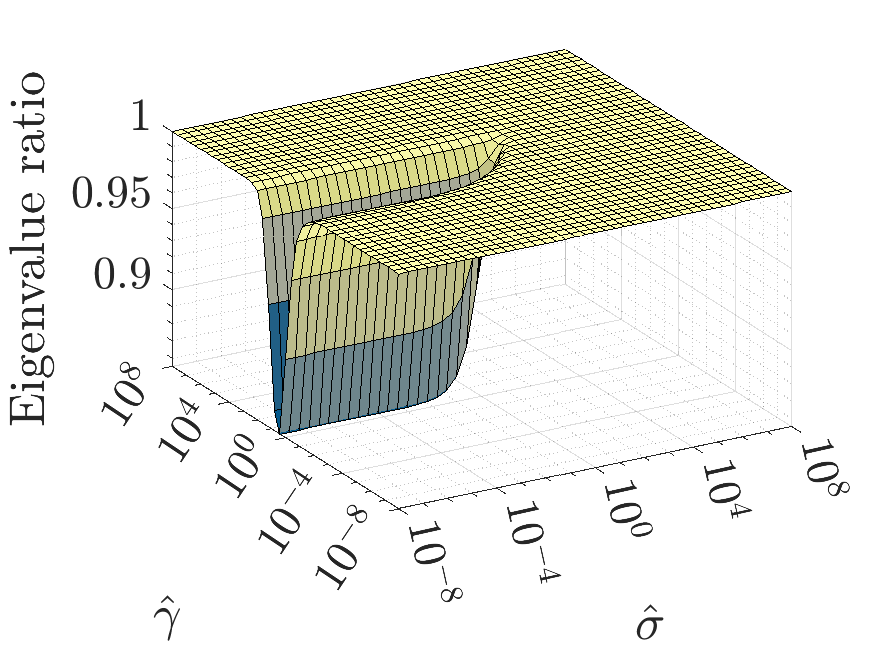}
        \caption{
            Tracking ($\alpha=-1$)
            \vspace{-.3cm}
        }
        \label{fig:scale:tau:tr-mDT}
    \end{subfigure}
    \hfill
    \begin{subfigure}[b]{.3\textwidth}
        \includegraphics[width=\textwidth]{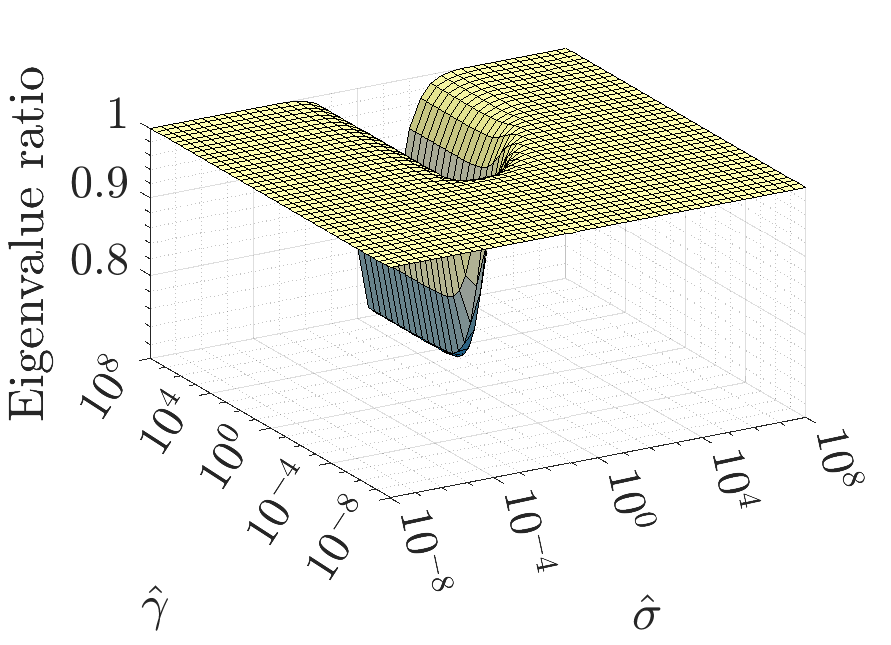}
        \caption{
            Terminal cost ($\alpha\rightarrow0$)
            \vspace{-.3cm}
        }
        \label{fig:scale:tau:tc-DT}
    \end{subfigure}
    \caption{
        Ratio $\abs{\theta(\Tidx=10^4)}/\abs{\theta(\Tidx=10^3)}$ of the preconditioned-eigenvalue magnitudes when scaling $\Tidx$ from $10^3$ to $10^4$ through $\tau$, for different preconditioners
    }
    \label{fig:scale:tau}
    \vspace{-.8cm}
\end{figure}

Keeping $T$ constant and scaling $\tau$ instead slightly complicates matters, as it changes not only $\Tidx$ but also $\widehat\sigma$ and $\widehat\gamma$.

\paragraph*{Tracking}
Using \textsc{Matlab}'s symbolic toolbox allows us to solve the limit
\begin{equation}
    \lim_{\Tidx\rightarrow\infty,\tau\rightarrow0,\tau\Tidx=T,\alpha=\pm1}\theta_{\{1,2\}} = 
        \frac12 + \frac{\tanh\bigl(\frac{T\sqrt{\gamma\sigma^2 + 1}}{2\sqrt\gamma}\bigr)^{-\alpha}(\sqrt\gamma\sigma \pm \iu)}{2\sqrt{\gamma\sigma^2 + 1}}    \mper
\end{equation}
This is a finite expression for both $\alpha=\pm1$ (the denominator cannot reach zero) and is non-zero as well (the real part of the numerator is always positive). \Cref{fig:scale:tau:tr-pDT,fig:scale:tau:tr-mDT} show that the eigenvalues stay almost constant when scaling a finite $\Tidx$ from $10^3$ to $10^4$, which we assumed implies weak scalability.

\paragraph*{Terminal cost}
This case is slightly simpler and can be computed by hand.
\begin{equation}
    \lim_{\Tidx\rightarrow\infty,\tau\rightarrow0,\tau\Tidx=T,\alpha\rightarrow0}\theta_1 = 1 + (1-\exp(-2\sigma T))/(\gamma\sigma)    \mcom
\end{equation}
which is again finite. It also cannot reach zero: the exponential has a negative argument (because $\sigma>0$), so both terms of the sum are positive. \Cref{fig:scale:tau:tc-DT} illustrates that the scaling translates well to finite $\Tidx$ values, again implying weak scalability.

\section{Numerical results} \label{sec:num}
    This section presents the results of numerical tests assessing the performance of our ParaDiag methods. \Cref{sec:num:track} first discusses tracking ParaDiag (\cref{alg:pd-track:new}), considering both $\alpha=1$ and the novel $\alpha=-1$ variant. \Cref{sec:num:tc} then moves on to the new terminal-cost method (\cref{alg:pd-tc:new}).

Our ParaDiag algorithms are tested with a \textsc{Matlab} code we call \texttt{pintopt}. Coding an efficient parallel ParaDiag implementation is a significant task \cite{caklovicParallelintimeCollocationMethod2023b} and is not the focus of the current paper. Hence, \texttt{pintopt} is sequential and not optimi\sz{}ed for speed, but rather serves as a readable and well-documented reference implementation that can be used to study iteration counts. The code is publicly available\footnote{The version of the \texttt{pintopt} \textsc{Matlab} package used here and code to reproduce our results are located at \url{https://gitlab.kuleuven.be/numa/public/pintopt}. New additions and bugfixes are tracked at \url{https://github.com/ArneBouillon/pintopt}.}.

All results use \textsc{gmres} as the iterative solver and are displayed in tables detailing the iteration counts for different parameter configurations. In each table, the rows investigate weak scaling, while the columns vary a different parameter such as the end time $T$ or the regulari\sz{}ation parameter $\gamma$. The tables at the left perform scaling of $\Tidx$ by increasing $T$ -- those on the right by decreasing $\tau$.

As a base problem, we study a parabolic diffusion equation, which is self-adjoint and dissipative such that our theoretical results are directly applicable. The problem involves a heat equation in two dimensions on the spatial domain $\Omega=[0,1]^2$. It reads
\begin{equation} \label{eq:num:intro:diff}
    \partial_t y = \Delta y + u
\end{equation}
with periodic boundary conditions and, in the case of tracking, a target trajectory
\begin{equation} \label{eq:num:intro:yd}
    y_\mathrm d(t, x) = \Bigl(\Bigl(12\pi^2+\frac1{12\pi^2\gamma}\Bigr)(t-T)-\Bigl(1+\frac1{(12\pi^2)^2\gamma}\Bigr)\Bigr)\sin(2\pi x_1)\sin(2\pi x_2)
\end{equation}
or, in the case of terminal cost, a target state
\begin{equation} \label{eq:num:intro:ytarget}
    y_\mathrm{target}(x) = \sin(2\pi x_1)\sin(2\pi x_2)    \mper
\end{equation}
This is the two-dimensional version of a problem studied in \cite{gotschelEfficientParallelinTimeMethod2019a}. In contrast to that paper, we use a non-smooth initial condition
\begin{equation} \label{eq:num:intro:yinit-rough}
    y_\mathrm{init}(x) = \frac1{12\pi^2\gamma}(1-T)\mathrm{sign}(\sin(2\pi x_1))\sin^2(2\pi x_2)    \mcom
\end{equation}
shown in Figure \ref{fig:num:intro:yinit-rough}.
\begin{figure}
    \centering
    \includegraphics[width=.5\textwidth]{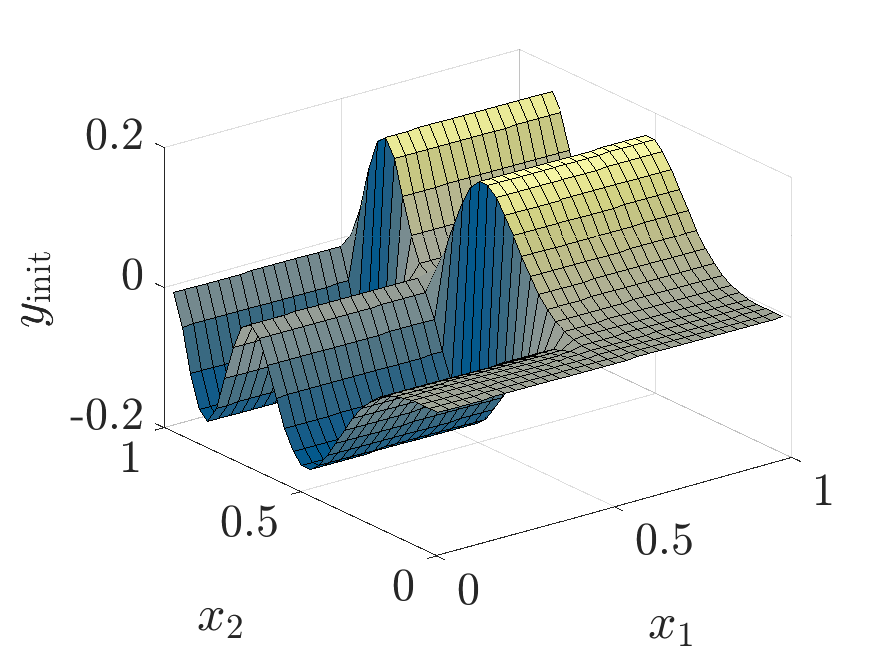}
    \label{fig:pd-track:results:y0:rough}
    \caption{
        Initial condition $y_\mathrm{init}$ from \cref{eq:num:intro:yinit-rough}
    }
    \label{fig:num:intro:yinit-rough}
\end{figure}
The choice for a non-smooth $y_\mathrm{init}$ is important, as a smooth initial condition leads to very fast convergence, as noticed in \cite{goddardNoteParallelPreconditioning2019,wuParallelInTimeBlockCirculantPreconditioner2020a}. We want to test our algorithms with a more challenging, non-smooth case.

Next to the self-adjoint equation \cref{eq:num:intro:diff} covered fully by this paper's analysis, we also consider a non-self-adjoint advection-diffusion equation that our new algorithms can solve, but for which we do not have theoretical results. Extending the previous equation with an advection term, consider
\begin{equation} \label{eq:num:intro:advdiff}
    \partial_t y = d\Delta y - \partial_{x_1}y - \partial_{x_2}y + u
\end{equation}
where $d\in\mathbb R$ controls the amount of diffusion and may vary. For this equation, we use the same $y_\mathrm d$, $y_\mathrm{target}$ and $y_\mathrm{init}$, given in \cref{eq:num:intro:yd,eq:num:intro:ytarget,eq:num:intro:yinit-rough}, as for the diffusion equation. Both \cref{eq:num:intro:diff,eq:num:intro:advdiff} are discreti\sz{}ed with $\Sidx=32\times32$ points in space and all spatial derivatives are discreti\sz{}ed with central differences.

    \subsection{Tracking} \label{sec:num:track}
    \begin{table}
    \footnotesize
    \begin{subtable}{.49\textwidth}
        \begin{tabular}{c|ccccc}
            $\Tidx\backslash T_\mathrm{ref}$ & \texttt{2e0} & \texttt{2e-1} & \texttt{2e-2} & \texttt{2e-3} & \texttt{2e-4}\\\hline
            \begin{filecontents*}{table-data/KsT-scalebyT.csv}
3/3,4/4,6/6,14/\textbf{8},$\varnothing$/\textbf{7}
3/3,3/3,5/5,10/\textbf{7},19/\textbf{8}
3/3,3/3,4/4,6/6,14/\textbf{8}
3/3,3/3,3/3,5/5,10/\textbf{7}
\end{filecontents*}
\csvreader[no head, late after line=\\]
            {table-data/KsT-scalebyT.csv}{}{\if1\thecsvrow30\else\if2\thecsvrow100\else\if3\thecsvrow300\else1000\fi\fi\fi & \csvcoli & \csvcolii & \csvcoliii & \csvcoliv & \csvcolv}
        \end{tabular}
        \caption{Scaling $T$, diffusion\vspace{-.1cm}}
        \label{tab:num:track:a}
    \end{subtable}
    \hfill
    \begin{subtable}{.49\textwidth}
        \begin{tabular}{c|ccccc}
            $\Tidx\backslash T$ & \texttt{2e0} & \texttt{2e-1} & \texttt{2e-2} & \texttt{2e-3} & \texttt{2e-4}\\\hline
            \begin{filecontents*}{table-data/KsT-scalebytau.csv}
3/3,4/4,6/6,14/\textbf{8},$\varnothing$/\textbf{7}
3/3,4/4,6/6,14/\textbf{8},24/\textbf{7}
3/3,4/4,6/6,14/\textbf{8},$\varnothing$/\textbf{7}
3/3,4/4,6/6,14/\textbf{8},$\varnothing$/\textbf{7}
\end{filecontents*}
\csvreader[no head, late after line=\\]
            {table-data/KsT-scalebytau.csv}{}{\if1\thecsvrow30\else\if2\thecsvrow100\else\if3\thecsvrow300\else1000\fi\fi\fi & \csvcoli & \csvcolii & \csvcoliii & \csvcoliv & \csvcolv}
        \end{tabular}
        \caption{Scaling $\tau$, diffusion\vspace{-.1cm}}
        \label{tab:num:track:b}
    \end{subtable}
    \vfill
    \begin{subtable}{.49\textwidth}
        \begin{tabular}{c|ccccc}
            $\Tidx\backslash \gamma$ & \texttt{5e-8} & \texttt{5e-5} & \texttt{5e-2} & \texttt{5e1} & \texttt{5e4}\\\hline
            \begin{filecontents*}{table-data/Ksg-scalebyT.csv}
2/2,4/4,3/3,3/3,2/2
2/2,4/4,3/3,3/3,2/2
2/2,4/4,3/3,3/3,2/2
2/2,3/3,3/3,3/3,2/2
\end{filecontents*}
\csvreader[no head, late after line=\\]
            {table-data/Ksg-scalebyT.csv}{}{\if1\thecsvrow30\else\if2\thecsvrow100\else\if3\thecsvrow300\else1000\fi\fi\fi & \csvcoli & \csvcolii & \csvcoliii & \csvcoliv & \csvcolv}
        \end{tabular}
        \caption{Scaling $T$, diffusion\vspace{-.1cm}}
        \label{tab:num:track:c}
    \end{subtable}
    \hfill
    \begin{subtable}{.49\textwidth}
        \begin{tabular}{c|ccccc}
            $\Tidx\backslash \gamma$ & \texttt{5e-8} & \texttt{5e-5} & \texttt{5e-2} & \texttt{5e1} & \texttt{5e4}\\\hline
            \begin{filecontents*}{table-data/Ksg-scalebytau.csv}
2/2,4/4,3/3,3/3,2/2
2/2,5/5,3/3,3/3,3/3
3/3,7/7,3/3,3/3,3/3
3/3,8/8,3/3,3/3,3/3
\end{filecontents*}
\csvreader[no head, late after line=\\]
            {table-data/Ksg-scalebytau.csv}{}{\if1\thecsvrow30\else\if2\thecsvrow100\else\if3\thecsvrow300\else1000\fi\fi\fi & \csvcoli & \csvcolii & \csvcoliii & \csvcoliv & \csvcolv}
        \end{tabular}
        \caption{Scaling $\tau$, diffusion\vspace{-.1cm}}
        \label{tab:num:track:d}
    \end{subtable}
    \vfill
    \begin{subtable}{.49\textwidth}
        \begin{tabular}{c|ccccc}
            $\Tidx\backslash T_\mathrm{ref}$ & \texttt{2e0} & \texttt{2e-1} & \texttt{2e-2} & \texttt{2e-3} & \texttt{2e-4}\\\hline
            \begin{filecontents*}{table-data/KsdT-scalebyT.csv}
5/5,\textbf{7}/8,19/\textbf{10},$\varnothing$/\textbf{7},$\varnothing$/\textbf{4}
5/5,7/7,11/11,$\varnothing$/\textbf{9},$\varnothing$/\textbf{5}
5/5,6/6,8/8,19/\textbf{10},$\varnothing$/\textbf{7}
5/5,6/6,7/7,11/11,$\varnothing$/\textbf{9}
\end{filecontents*}
\csvreader[no head, late after line=\\]
            {table-data/KsdT-scalebyT.csv}{}{\if1\thecsvrow30\else\if2\thecsvrow100\else\if3\thecsvrow300\else1000\fi\fi\fi & \csvcoli & \csvcolii & \csvcoliii & \csvcoliv & \csvcolv}
        \end{tabular}
        \caption{Scaling $T$, advection-diffusion\vspace{-.1cm}}
        \label{tab:num:track:e}
    \end{subtable}
    \hfill
    \begin{subtable}{.49\textwidth}
        \begin{tabular}{c|ccccc}
            $\Tidx\backslash T$ & \texttt{2e0} & \texttt{2e-1} & \texttt{2e-2} & \texttt{2e-3} & \texttt{2e-4}\\\hline
            \begin{filecontents*}{table-data/KsdT-scalebytau.csv}
5/5,\textbf{7}/8,19/\textbf{10},$\varnothing$/\textbf{7},$\varnothing$/\textbf{4}
6/6,\textbf{7}/8,19/\textbf{10},$\varnothing$/\textbf{7},$\varnothing$/\textbf{4}
6/6,8/8,19/\textbf{10},$\varnothing$/\textbf{7},$\varnothing$/\textbf{4}
6/6,8/8,19/\textbf{10},$\varnothing$/\textbf{7},$\varnothing$/\textbf{4}
\end{filecontents*}
\csvreader[no head, late after line=\\]
            {table-data/KsdT-scalebytau.csv}{}{\if1\thecsvrow30\else\if2\thecsvrow100\else\if3\thecsvrow300\else1000\fi\fi\fi & \csvcoli & \csvcolii & \csvcoliii & \csvcoliv & \csvcolv}
        \end{tabular}
        \caption{Scaling $\tau$, advection-diffusion\vspace{-.1cm}}
        \label{tab:num:track:f}
    \end{subtable}
    \vfill
    \begin{subtable}{.49\textwidth}
        \begin{tabular}{c|ccccc}
            $\Tidx\backslash \gamma$ & \texttt{5e-8} & \texttt{5e-5} & \texttt{5e-2} & \texttt{5e1} & \texttt{5e4}\\\hline
            \begin{filecontents*}{table-data/Ksdg-scalebyT.csv}
2/2,4/4,5/5,4/4,3/3
2/2,4/4,5/5,4/4,3/3
2/2,3/3,5/5,3/3,3/3
2/2,3/3,5/5,3/3,3/3
\end{filecontents*}
\csvreader[no head, late after line=\\]
            {table-data/Ksdg-scalebyT.csv}{}{\if1\thecsvrow30\else\if2\thecsvrow100\else\if3\thecsvrow300\else1000\fi\fi\fi & \csvcoli & \csvcolii & \csvcoliii & \csvcoliv & \csvcolv}
        \end{tabular}
        \caption{Scaling $T$, advection-diffusion\vspace{-.1cm}}
        \label{tab:num:track:g}
    \end{subtable}
    \hfill
    \begin{subtable}{.49\textwidth}
        \begin{tabular}{c|ccccc}
            $\Tidx\backslash \gamma$ & \texttt{5e-8} & \texttt{5e-5} & \texttt{5e-2} & \texttt{5e1} & \texttt{5e4}\\\hline
            \begin{filecontents*}{table-data/Ksdg-scalebytau.csv}
2/2,4/4,5/5,4/4,3/3
2/2,6/6,6/6,4/4,3/3
3/3,9/9,6/6,4/4,3/3
3/3,11/11,6/6,3/3,3/3
\end{filecontents*}
\csvreader[no head, late after line=\\]
            {table-data/Ksdg-scalebytau.csv}{}{\if1\thecsvrow30\else\if2\thecsvrow100\else\if3\thecsvrow300\else1000\fi\fi\fi & \csvcoli & \csvcolii & \csvcoliii & \csvcoliv & \csvcolv}
        \end{tabular}
        \caption{Scaling $\tau$, advection-diffusion\vspace{-.1cm}}
        \label{tab:num:track:h}
    \end{subtable}
    \vfill
    \begin{subtable}{.49\textwidth}
        \begin{tabular}{c|ccccc}
            $\Tidx\backslash d$ & \texttt{1e-3} & \texttt{1e-2} & \texttt{1e-1} & \texttt{1e0} & \texttt{1e1}\\\hline
            \begin{filecontents*}{table-data/Ksdd-scalebyT.csv}
7/7,7/7,5/5,3/3,2/2
7/7,7/7,5/5,3/3,2/2
7/7,7/7,5/5,3/3,2/2
7/7,6/6,5/5,3/3,2/2
\end{filecontents*}
\csvreader[no head, late after line=\\]
            {table-data/Ksdd-scalebyT.csv}{}{\if1\thecsvrow30\else\if2\thecsvrow100\else\if3\thecsvrow300\else1000\fi\fi\fi & \csvcoli & \csvcolii & \csvcoliii & \csvcoliv & \csvcolv}
        \end{tabular}
        \caption{Scaling $T$, advection-diffusion\vspace{-.4cm}}
        \label{tab:num:track:i}
    \end{subtable}
    \hfill
    \begin{subtable}{.49\textwidth}
        \begin{tabular}{c|ccccc}
            $\Tidx\backslash d$ & \texttt{1e-3} & \texttt{1e-2} & \texttt{1e-1} & \texttt{1e0} & \texttt{1e1}\\\hline
            \begin{filecontents*}{table-data/Ksdd-scalebytau.csv}
7/7,7/7,5/5,3/3,2/2
8/8,7/7,6/6,3/3,3/3
8/8,7/7,6/6,3/3,3/3
8/8,8/8,6/6,3/3,3/3
\end{filecontents*}
\csvreader[no head, late after line=\\]
            {table-data/Ksdd-scalebytau.csv}{}{\if1\thecsvrow30\else\if2\thecsvrow100\else\if3\thecsvrow300\else1000\fi\fi\fi & \csvcoli & \csvcolii & \csvcoliii & \csvcoliv & \csvcolv}
        \end{tabular}
        \caption{Scaling $\tau$, advection-diffusion\vspace{-.4cm}}
        \label{tab:num:track:j}
    \end{subtable}
    \caption{\textsc{gmres} iteration counts ($\alpha=1$/$\alpha=-1$) for tracking ParaDiag applied to the diffusion equation \cref{eq:num:intro:diff} or the advection-diffusion equation \cref{eq:num:intro:advdiff}. The symbol $\varnothing$ indicates a failure to converge within 25 iterations. When one $\alpha$ value outperforms the other, it is bold-faced. By default, $T_\mathrm{ref}=2$, $\gamma=0.05$, and $d=0.1$ when applicable.\vspace{-.5cm}}
    \label{tab:num:track}
\end{table}

The results from applying ParaDiag to the (advection-)diffusion example are listed in \cref{tab:num:track}. The case $\alpha=-1$ outperforms $\alpha=1$ when $T$ is small and does not make much difference otherwise, as observed in \cref{sec:pd-track:alpha}. A wide variety of $\gamma$ values is tested, all resulting in very reasonable iteration counts.

There are two more significant observations. Firstly, the iteration count mostly stays constant when increasing $\Tidx$, which is the weak scalability theori\sz{}ed in \cref{sec:scale}. Secondly, the advection-diffusion case is comparable to the pure diffusion equation, both in scaling and in the effect of $\alpha$. While decreasing the amount of diffusion $d$ increases the iteration count, scaling remains good, as observed in \cref{tab:num:track:j}. This suggests that the conclusions from our self-adjoint study may apply more broadly to non-self-adjoint problems as well. Studying how well optimi\sz{}ation ParaDiag performs for hyperbolic problems such as the pure advection case is left as future work. We remark that, for advection-dominated problems, a carefully selected spatial discreti\sz{}ation (potentially using stabili\sz{}ation \cite{guermondStabilizationGalerkinApproximations1999}) is vital for an accurate solution.

    \subsection{Terminal cost} \label{sec:num:tc}
    \begin{table}
    \footnotesize
    \begin{subtable}{.49\textwidth}
        \begin{tabular}{c|ccccc}
            $\Tidx\backslash T_\mathrm{ref}$ & \texttt{2e0} & \texttt{2e-1} & \texttt{2e-2} & \texttt{2e-3} & \texttt{2e-4}\\\hline
            \begin{filecontents*}{table-data/KsT-tc-scalebyT.csv}
2,3,4,4,3
2,2,3,4,3
3,2,3,4,4
3,2,2,3,4
\end{filecontents*}
\csvreader[no head, late after line=\\]
            {table-data/KsT-tc-scalebyT.csv}{}{\if1\thecsvrow30\else\if2\thecsvrow100\else\if3\thecsvrow300\else1000\fi\fi\fi & \csvcoli & \csvcolii & \csvcoliii & \csvcoliv & \csvcolv}
        \end{tabular}
        \caption{Scaling $T$, diffusion}
        \label{tab:num:tc:a}
    \end{subtable}
    \hfill
    \begin{subtable}{.49\textwidth}
        \begin{tabular}{c|ccccc}
            $\Tidx\backslash T$ & \texttt{2e0} & \texttt{2e-1} & \texttt{2e-2} & \texttt{2e-3} & \texttt{2e-4}\\\hline
            \begin{filecontents*}{table-data/KsT-tc-scalebytau.csv}
2,3,4,4,3
2,3,4,4,3
2,3,4,4,3
2,3,4,4,3
\end{filecontents*}
\csvreader[no head, late after line=\\]
            {table-data/KsT-tc-scalebytau.csv}{}{\if1\thecsvrow30\else\if2\thecsvrow100\else\if3\thecsvrow300\else1000\fi\fi\fi & \csvcoli & \csvcolii & \csvcoliii & \csvcoliv & \csvcolv}
        \end{tabular}
        \caption{Scaling $\tau$, diffusion}
        \label{tab:num:tc:b}
    \end{subtable}
    \vfill
    \begin{subtable}{.49\textwidth}
        \begin{tabular}{c|ccccc}
            $\Tidx\backslash \gamma$ & \texttt{5e-8} & \texttt{5e-5} & \texttt{5e-2} & \texttt{5e1} & \texttt{5e4}\\\hline
            \begin{filecontents*}{table-data/Ksg-tc-scalebyT.csv}
3,3,2,1,1
12,3,2,1,1
12,3,3,1,1
$\varnothing$,3,3,1,1
\end{filecontents*}
\csvreader[no head, late after line=\\]
            {table-data/Ksg-tc-scalebyT.csv}{}{\if1\thecsvrow30\else\if2\thecsvrow100\else\if3\thecsvrow300\else1000\fi\fi\fi & \csvcoli & \csvcolii & \csvcoliii & \csvcoliv & \csvcolv}
        \end{tabular}
        \caption{Scaling $T$, diffusion}
        \label{tab:num:tc:c}
    \end{subtable}
    \hfill
    \begin{subtable}{.49\textwidth}
        \begin{tabular}{c|ccccc}
            $\Tidx\backslash \gamma$ & \texttt{5e-8} & \texttt{5e-5} & \texttt{5e-2} & \texttt{5e1} & \texttt{5e4}\\\hline
            \begin{filecontents*}{table-data/Ksg-tc-scalebytau.csv}
3,3,2,1,1
11,3,2,1,1
11,3,2,1,1
10,3,2,1,1
\end{filecontents*}
\csvreader[no head, late after line=\\]
            {table-data/Ksg-tc-scalebytau.csv}{}{\if1\thecsvrow30\else\if2\thecsvrow100\else\if3\thecsvrow300\else1000\fi\fi\fi & \csvcoli & \csvcolii & \csvcoliii & \csvcoliv & \csvcolv}
        \end{tabular}
        \caption{Scaling $\tau$, diffusion}
        \label{tab:num:tc:d}
    \end{subtable}
    \vfill
    \begin{subtable}{.49\textwidth}
        \begin{tabular}{c|ccccc}
            $\Tidx\backslash T_\mathrm{ref}$ & \texttt{2e0} & \texttt{2e-1} & \texttt{2e-2} & \texttt{2e-3} & \texttt{2e-4}\\\hline
            \begin{filecontents*}{table-data/KsdT-tc-scalebyT.csv}
7,7,5,4,3
5,6,6,4,3
5,5,6,5,4
5,5,6,6,4
\end{filecontents*}
\csvreader[no head, late after line=\\]
            {table-data/KsdT-tc-scalebyT.csv}{}{\if1\thecsvrow30\else\if2\thecsvrow100\else\if3\thecsvrow300\else1000\fi\fi\fi & \csvcoli & \csvcolii & \csvcoliii & \csvcoliv & \csvcolv}
        \end{tabular}
        \caption{Scaling $T$, advection-diffusion}
        \label{tab:num:tc:e}
    \end{subtable}
    \hfill
    \begin{subtable}{.49\textwidth}
        \begin{tabular}{c|ccccc}
            $\Tidx\backslash T$ & \texttt{2e0} & \texttt{2e-1} & \texttt{2e-2} & \texttt{2e-3} & \texttt{2e-4}\\\hline
            \begin{filecontents*}{table-data/KsdT-tc-scalebytau.csv}
7,7,5,4,3
6,7,5,4,3
5,6,5,4,3
6,6,5,4,3
\end{filecontents*}
\csvreader[no head, late after line=\\]
            {table-data/KsdT-tc-scalebytau.csv}{}{\if1\thecsvrow30\else\if2\thecsvrow100\else\if3\thecsvrow300\else1000\fi\fi\fi & \csvcoli & \csvcolii & \csvcoliii & \csvcoliv & \csvcolv}
        \end{tabular}
        \caption{Scaling $\tau$, advection-diffusion}
        \label{tab:num:tc:f}
    \end{subtable}
    \vfill
    \begin{subtable}{.49\textwidth}
        \begin{tabular}{c|ccccc}
            $\Tidx\backslash \gamma$ & \texttt{5e-8} & \texttt{5e-5} & \texttt{5e-2} & \texttt{5e1} & \texttt{5e4}\\\hline
            \begin{filecontents*}{table-data/Ksdg-tc-scalebyT.csv}
7,7,7,2,1
$\varnothing$,5,5,2,1
$\varnothing$,5,5,2,1
$\varnothing$,5,5,2,1
\end{filecontents*}
\csvreader[no head, late after line=\\]
            {table-data/Ksdg-tc-scalebyT.csv}{}{\if1\thecsvrow30\else\if2\thecsvrow100\else\if3\thecsvrow300\else1000\fi\fi\fi & \csvcoli & \csvcolii & \csvcoliii & \csvcoliv & \csvcolv}
        \end{tabular}
        \caption{Scaling $T$, advection-diffusion}
        \label{tab:num:tc:g}
    \end{subtable}
    \hfill
    \begin{subtable}{.49\textwidth}
        \begin{tabular}{c|ccccc}
            $\Tidx\backslash \gamma$ & \texttt{5e-8} & \texttt{5e-5} & \texttt{5e-2} & \texttt{5e1} & \texttt{5e4}\\\hline
            \begin{filecontents*}{table-data/Ksdg-tc-scalebytau.csv}
7,7,7,2,1
7,7,6,2,1
7,7,5,2,1
7,7,6,2,1
\end{filecontents*}
\csvreader[no head, late after line=\\]
            {table-data/Ksdg-tc-scalebytau.csv}{}{\if1\thecsvrow30\else\if2\thecsvrow100\else\if3\thecsvrow300\else1000\fi\fi\fi & \csvcoli & \csvcolii & \csvcoliii & \csvcoliv & \csvcolv}
        \end{tabular}
        \caption{Scaling $\tau$, advection-diffusion}
        \label{tab:num:tc:h}
    \end{subtable}
    \vfill
    \begin{subtable}{.49\textwidth}
        \begin{tabular}{c|ccccc}
            $\Tidx\backslash d$ & \texttt{1e-3} & \texttt{1e-2} & \texttt{1e-1} & \texttt{1e0} & \texttt{1e1}\\\hline
            \begin{filecontents*}{table-data/Ksdd-tc-scalebyT.csv}
9,8,7,5,4
8,7,5,5,4
5,5,5,5,4
5,5,5,5,4
\end{filecontents*}
\csvreader[no head, late after line=\\]
            {table-data/Ksdd-tc-scalebyT.csv}{}{\if1\thecsvrow30\else\if2\thecsvrow100\else\if3\thecsvrow300\else1000\fi\fi\fi & \csvcoli & \csvcolii & \csvcoliii & \csvcoliv & \csvcolv}
        \end{tabular}
        \caption{Scaling $T$, advection-diffusion}
        \label{tab:num:tc:i}
    \end{subtable}
    \hfill
    \begin{subtable}{.49\textwidth}
        \begin{tabular}{c|ccccc}
            $\Tidx\backslash d$ & \texttt{1e-3} & \texttt{1e-2} & \texttt{1e-1} & \texttt{1e0} & \texttt{1e1}\\\hline
            \begin{filecontents*}{table-data/Ksdd-tc-scalebytau.csv}
9,8,7,5,4
9,7,6,5,4
11,7,5,5,4
11,8,6,3,3
\end{filecontents*}
\csvreader[no head, late after line=\\]
            {table-data/Ksdd-tc-scalebytau.csv}{}{\if1\thecsvrow30\else\if2\thecsvrow100\else\if3\thecsvrow300\else1000\fi\fi\fi & \csvcoli & \csvcolii & \csvcoliii & \csvcoliv & \csvcolv}
        \end{tabular}
        \caption{Scaling $\tau$, advection-diffusion}
        \label{tab:num:tc:j}
    \end{subtable}
    \caption{\textsc{gmres} iteration counts for terminal-cost ParaDiag applied to the diffusion equation \cref{eq:num:intro:diff} or the advection-diffusion equation \cref{eq:num:intro:advdiff}. The symbol $\varnothing$ indicates a failure to converge within 25 iterations. By default, $T_\mathrm{ref}=2$, $\gamma=0.05$, and $d=0.1$ when applicable. All results use $\alpha=10^{-4}$.}
    \label{tab:num:tc}
\end{table}

The same experiments were done for terminal-cost objectives in \cref{tab:num:tc}, although $\alpha=10^{-4}$ was chosen here. The results are very promising. For the diffusion equation, iteration counts are low across the board, with no scenario surpassing $4$ iterations. The only exception is that of a \emph{very} small $\gamma$ (that is, a large $\widehat\gamma$), which was indeed theori\sz{}ed to work poorly in \cref{sec:pd-tc:interp}.

When adding advection, slightly more iterations are needed, but the increase stays reasonable. Again, the qualitative insights from the self-adjoint case carry over: a small regulari\sz{}ation parameter $\gamma$ can cause slow convergence. Small $d$ values also seem to result in an increased iteration count. However, scaling is excellent in all scenarios, confirming the theoretical conclusions from \cref{sec:scale}.

\section{Conclusions} \label{sec:concl}
    This paper has extended optimi\sz{}ation ParaDiag in three ways. For the existing algorithm \cite{wuDiagonalizationbasedParallelintimeAlgorithms2020b}, aimed at tracking objectives, we proposed an alpha-circulant extension to improve the edge case of the regime with small final time $T$ and a generali\sz{}ation to non-self-adjoint problems. We also designed a new algorithm to treat terminal-cost objectives, which is robust with respect to changing $T$. In doing so, we greatly expanded the range of problems for which efficient ParaDiag algorithms are available.

Secondly, we were able to formulate a precise expression for the preconditioned eigenvalues of all optimi\sz{}ation ParaDiag methods, in the self-adjoint case. This significantly improves our understanding of these algorithms, for which very little theory was available before. We used this knowledge for two purposes. \begin{itemize}
    \item For dissipative, self-adjoint equations with a tracking objective, and when using the new parameter $\alpha=-1$ to construct a preconditioner, we were able to prove a guaranteed \textsc{gmres} convergence factor of $1/2$.
    \item In a theoretical parallel-scaling analysis, we conjectured good weak scalability of all ParaDiag variants in the limit for many time steps.
\end{itemize}
This scalability was confirmed by numerical experiments that used \textsc{gmres} iteration counts as an indicator of performance. In addition, these tests suggested the theoretical conclusions carry over to the non-self-adjoint case, even though our analysis does not apply there.

As a third contribution, our progress clears the way for exciting research in the future. With a more general method, into which some theoretical insight is available, potential next steps include non-linear ParaDiag algorithms, such as those already proposed for \textsc{ivp} ParaDiag \cite{ganderTimeParallelizationNonlinear2017a,liuFastBlockAcirculant2020a}. Here, the robustness of the $\alpha=-1$ choice for tracking could prove important in dealing with low-$\widehat\sigma$ lineari\sz{}ations.

Another interesting avenue for future work building on our results is the study of different time-discreti\sz{}ation methods -- especially if they can be written as \cref{eq:pd-track:anal:resc} or \cref{eq:pd-tc:anal:resc} since then, our theoretical results apply. In addition, ParaDiag has shown promise for hyperbolic problems, which are often challenging for time-parallel methods. Next steps in this context could include a study of the methods in this paper for advection equations, or improvements to optimi\sz{}ation ParaDiag for wave equations \cite{wuParallelInTimeBlockCirculantPreconditioner2020a} similar to those in this paper. On the computational side, different techniques have been applied to solve the smaller systems in the inversion procedure for \textsc{ivp} ParaDiag more efficiently \cite{liuROMacceleratedParallelintimePreconditioner2020,heVankatypeMultigridSolver2022a}. Adapting these methods to optimi\sz{}ation ParaDiag could substantially improve performance.

We lastly mention some alternatives to the preconditioners proposed in this paper that may be worthwhile to pursue. First, a very interesting recent result \cite{kressnerImprovedParallelintimeIntegration2022} in the domain of \textsc{ivp} ParaDiag suggests using alpha-circulant approximations, but not as preconditioners. Instead, it is noted that in the \textsc{ivp} situation, the exact system matrix is $\kronm P(\alpha)$ with $\alpha=0$ and its inversion is seen as an interpolation problem, with as data points several inversions with $\alpha_j\ne0$. Our terminal-cost preconditioner \cref{eq:pd-tc:new:Palpha} is not suitable for this, as it has $\kronm P(0)\ne\kronm A$. However, for the tracking preconditioner \cref{eq:pd-track:alpha:Palpha}, $\kronm P(0) = \kronm{\widehat A}$ does hold. As this text makes $\alpha\ne1$ feasible for tracking, we can use different $\alpha_j$ with magnitude $1$ as data points and \cite{kressnerImprovedParallelintimeIntegration2022}'s technique could now apply to tracking-type optimal-control ParaDiag. Second, for the terminal-cost case, alternative preconditioners (especially those that retain the $E$ block in \cref{eq:pd-tc:new:aao}) may improve on the convergence and scaling of our proposal in the regime with low $\widehat\sigma$ and high $\widehat\gamma$.

\clearpage
\appendix
\section{Proof of Theorem \ref*{thm:pd-track:eigs}} \label{sec:apdx-pdte}
    We denote by $C(\alpha)$ the top-left block of the matrix inverted in \cref{eq:thm:pd-track:eigs:M}. We start by switching the top and bottom halves of the rows of both $R$ and the inverted matrix -- which does not change $M$ -- and applying the block matrix inversion property from \cite[page 44]{bernsteinMatrixMathematicsTheory2005}, giving
\begingroup\allowdisplaybreaks
\begin{align*}
    M &= {\begin{litmat}
        -\psi I & C(\alpha)^\trsp\\C(\alpha) & \psi I\\
    \end{litmat}^{-1}\begin{litmat}[cc|cc]
        &&&\\
        &&\alpha\varphi\\
        \cmidrule(lr){1-2}\cmidrule(lr){3-4}
        &\alpha\varphi&\\
        &&\\
    \end{litmat}}\\
    &={\left[\begin{smallmatrix}
        (-\psi I - \frac1\psi C(\alpha)^\trsp C(\alpha))^{-1}\\
        & (\psi I + \frac1\psi C(\alpha)C(\alpha)^\trsp)^{-1}\\
    \end{smallmatrix}\right]\left[\begin{smallmatrix}
        I & -\frac1\psi C(\alpha)^\trsp\\
        \frac1\psi C(\alpha) & I\\
    \end{smallmatrix}\right]\begin{litmat}[cc|cc]
        &&&\\
        &&\alpha\varphi\\
        \cmidrule(lr){1-2}\cmidrule(lr){3-4}
        &\alpha\varphi&\\
        &&\\
    \end{litmat}}\\
    &= {\begin{litmat}
        \overbrace{(-\psi I - \frac1\psi C(\alpha)^\trsp C(\alpha))^{-1}}^{= -H}\\
        & \underbrace{(\psi I + \frac1\psi C(\alpha)C(\alpha)^\trsp)^{-1}}_{\eqqcolon H}\\
    \end{litmat}\left[\begin{smallarray}{ccc|ccc}
        &&-\frac{\alpha\varphi}\psi&&&\\
        &&&\\
        &&\invisibleminus\frac{\varphi^2}\psi\alpha^2&\alpha\varphi\\
        \cmidrule(lr){1-3}\cmidrule(lr){4-6}&&\alpha\varphi&-\frac{\varphi^2}\psi\alpha^2\\
        &&&\\
        &&&\frac{\alpha\varphi}\psi\\
    \end{smallarray}\right]}    \mcom
\end{align*}
\endgroup
where we know $\alpha^2=1$. The fact that $M$ has only two potentially non-zero eigenvalues is clear: the second matrix in the product above has rank $2$, such that the result of a multiplication by it cannot have any higher rank.

To find out more about these non-zero eigenvalues, first observe that 
\begin{equation}
C(\alpha)^\trsp C(\alpha)=C(\alpha)C(\alpha)^\trsp    \mcom
\end{equation}
as can be easily checked. This justifies the use of the variable $H$ for both blocks above. We will try to calculate $H$ later, but in the spirit of not doing excess work, let us first see which parts of $H$ we need at all.
\begin{equation}
    M = \begin{litmat}
        -H\\&H\\
    \end{litmat}\begin{litmat}[ccc|ccc]
        &&-\frac{\alpha\varphi}\psi&&&\\
        &&&\\
        &&\frac{\varphi^2}\psi&\alpha\varphi\\
        \cmidrule(lr){1-3}\cmidrule(lr){4-6}&&\alpha\varphi&-\frac{\varphi^2}\psi\\
        &&&\\
        &&&\frac{\alpha\varphi}\psi\\
    \end{litmat} \eqqcolon \begin{litmat}[ccc|ccc]
        &&a_1&b_1&&\\
        &&\vdots&\vdots\\
        &&a_{\widehat\Tidx}&b_{\widehat\Tidx}\\
        \cmidrule(lr){1-3}\cmidrule(lr){4-6} &&a_{\widehat\Tidx+1}&b_{\widehat\Tidx+1}\\
        &&\vdots&\vdots\\
        &&a_{2\widehat\Tidx}&b_{2\widehat\Tidx}\\
    \end{litmat}    \mcom
\end{equation}
which means that $M$'s non-zero eigenvalues are the same as those of its middle block
\begin{equation} \label{eq:pd-track:eigs:Mred}
    M_\mathrm{red} = \begin{litmat}
        a_{\widehat\Tidx}&b_{\widehat\Tidx}\\
        a_{\widehat\Tidx+1}&b_{\widehat\Tidx+1}\\
    \end{litmat} = \begin{litmat}
        \frac{\alpha\varphi}\psi h_{\en,0} - \frac{\varphi^2}\psi h_{\en,\en}   &   -\alpha\varphi h_{\en,\en}\\
        \alpha\varphi h_{0,0}   &    -\frac{\varphi^2}\psi h_{0,0} + \frac{\alpha\varphi}\psi h_{0,\en}
    \end{litmat}    \mper
\end{equation}

Thus, it suffices to find the corner values of
\begin{equation} \label{eq:pd-track:anal:H}
\begin{aligned}
    H &= (\psi I + \frac1\psi C(\alpha)C(\alpha)^\trsp)^{-1} =\psi(\underbrace{\psi^2I+C(\alpha)C(\alpha)^\trsp}_{\eqqcolon G})^{-1}\\&= \psi\begin{litmat}
        1+\varphi^2+\psi^2 & -\varphi & & -\alpha\varphi\\
        -\varphi & 1+\varphi^2+\psi^2 & \ddots &\\
        & \ddots & \ddots & -\varphi\\
        -\alpha\varphi & & -\varphi & 1+\varphi^2+\psi^2\\
    \end{litmat}^{-1}    \mper
\end{aligned}
\end{equation}
The matrix $G$ being inverted is $\alpha$-circulant and symmetric -- qualities that are maintained by the inversion. Then $h_{0,0}=h_{\en,\en}$ and $h_{0,\en}=h_{\en,0}$. Hence
\begin{equation} \label{eq:thm:pd-track:eigs:abeq}
    a_{\widehat\Tidx} = b_{\widehat\Tidx+1} = -\frac{\varphi^2}\psi h_{0,0} + \frac{\alpha\varphi}\psi h_{0,\en} \quad \text{and} \quad a_{\widehat\Tidx+1} = -b_{\widehat\Tidx} = \alpha\varphi h_{0,0}    \mper
\end{equation}
This means that
\begin{equation} \label{eq:thm:pd-track:eigs:Mredeig}
    \mathrm{eig}(M_\mathrm{red}) = a_{\widehat\Tidx} \pm b_{\widehat\Tidx}\iu    \mper
\end{equation}

All this assumes we have inverted $G$. Alpha-circulant matrices can be inverted using their spectral decomposition \cref{eq:pd-track:alpha:Calpha-fact}. Up until now, we used this as a computational tool; extracting useful analytical expressions is not a trivial feat. The diagonali\sz{}ation reads $G=\Gamma_\alpha^{-1}\dfm ^*D\dfm \Gamma_\alpha \Leftrightarrow G^{-1}=\Gamma_\alpha^{-1}\dfm ^*D^{-1}\dfm \Gamma_\alpha$. Here,
\begin{equation}
\begin{aligned}
    D &= \mathrm{diag}(\sqrt{\widehat\Tidx}\dfm \Gamma_\alpha\mv{g_1})\\
    &= \mathrm{diag}\{
        1+\varphi^2+\psi^2 - \E^{(j/\widehat\Tidx)2\pi\iu}\alpha^{1/\widehat\Tidx}\varphi - \E^{-(j/\widehat\Tidx)2\pi\iu}\alpha^{(\widehat\Tidx-1)/\widehat\Tidx}\alpha\varphi
    \}_{j=0}^{\widehat\Tidx-1}\\
    &= \mathrm{diag}\left\{d(\beta_j(\alpha,\widehat\Tidx))\right\}_{j=0}^{\widehat\Tidx-1}
\end{aligned}
\end{equation}
where $\mv{g_1}$ denotes $G$'s first column and where we defined
\begin{equation}
    d(\beta) = 1+\varphi^2+\psi^2 - 2\varphi\cos\beta \quad \text{and} \quad \beta_j(\alpha,\widehat\Tidx) = \begin{cases*}
        2j\pi/\widehat\Tidx & if $\alpha=1$\\
        2(j+\frac12)\pi/\widehat\Tidx & if $\alpha=-1\mper$
    \end{cases*}
\end{equation}
Then from $H=\psi G^{-1}$ follows, noting the definitions of $\Gamma_\alpha$ and $\dfm $,
\begin{subequations} \label{eq:pd-track:hs-cos}
\begin{align}
    \label{eq:pd-track:hs-cos:1} h_{0,0}(\alpha,\widehat\Tidx)=h_{\en,\en}(\alpha,\widehat\Tidx)&=\frac\psi{\widehat\Tidx}\sum_{j=0}^{\widehat\Tidx-1}{d(\beta_j(\alpha,\widehat\Tidx))^{-1}}    \mcom\\
    \label{eq:pd-track:hs-cos:2} \begin{split}h_{0,\en}(\alpha,\widehat\Tidx)=h_{\en,0}(\alpha,\widehat\Tidx)&=\frac\psi{\widehat\Tidx}\sum_{j=0}^{\widehat\Tidx-1}{\alpha^{(\widehat\Tidx-1)/\widehat\Tidx}\E^{-(j/\widehat\Tidx)2\pi\iu}d(\beta_j(\alpha,\widehat\Tidx))^{-1}}\\&= \alpha\frac\psi{\widehat\Tidx}\sum_{j=0}^{\widehat\Tidx-1}{\E^{-\beta_j(\alpha,\widehat\Tidx)\iu}d(\beta_j(\alpha,\widehat\Tidx))^{-1}}   \mper \end{split}
\end{align}
\end{subequations}
These analytic expressions are not insightful. Luckily, we have yet another avenue to find $h_{0,0}$ and $h_{0,\en}$: \cite[Theorem 1(a)]{searleInvertingCirculantMatrices1979} offers explicit formulas for inverting certain three-element circulant matrices. For $\alpha=1$, those formulas mean that
\begin{subequations}
\begin{align}
    h_{0,0}(1,\widehat\Tidx) = h_{\en,\en}(1,\widehat\Tidx) &= \psi\frac{z_1z_2}{\varphi(z_2-z_1)}\bigg(\frac1{1-z_1^{\widehat\Tidx}}-\frac1{1-z_2^{\widehat\Tidx}}\bigg)    \mcom\\
    h_{0,\en}(1,\widehat\Tidx) = h_{\en,0}(1,\widehat\Tidx) &= \psi\frac{z_1z_2}{\varphi(z_2-z_1)}\bigg(\frac{z_1}{1-z_1^{\widehat\Tidx}}-\frac{z_2}{1-z_2^{\widehat\Tidx}}\bigg)
\end{align}
\end{subequations}
with $z_{\{1,2\}}=(1+\varphi^2+\psi^2\pm\sqrt{(1+\varphi^2+\psi^2)^2-4\varphi^2})/(2\varphi)$. However, \cite{searleInvertingCirculantMatrices1979} tells us nothing about the case $\alpha=-1$.

\begin{figure}[h!]
    \centering
    \begin{tikzpicture}[scale=0.9]
        \draw[-] (0,0) -- (12,0);
        \draw[shift={(0,0)},color=black] (0pt,5pt) -- (0pt,-5pt) node[below] {$0$};
        \draw[shift={(0,0)},color=matlabred] (0pt,18pt) node {$\beta_0(1,\widehat\Tidx)$};
        \draw[shift={(0,0)},color=black] (0pt,33pt) node {$\beta_0(1,2\widehat\Tidx)$};

        \draw[shift={(2,0)},color=black] (0pt,5pt) -- (0pt,-5pt) node[below] {$\pi/3$};
        \draw[shift={(2,0)},color=matlabblue] (0pt,18pt) node {$\beta_0(-1,\widehat\Tidx)$};
        \draw[shift={(2,0)},color=black] (0pt,33pt) node {$\beta_1(1,2\widehat\Tidx)$};

        \draw[shift={(4,0)},color=black] (0pt,5pt) -- (0pt,-5pt) node[below] {$2\pi/3$};
        \draw[shift={(4,0)},color=matlabred]  (0pt,18pt) node {$\beta_1(1,\widehat\Tidx)$};
        \draw[shift={(4,0)},color=black]  (0pt,33pt) node {$\beta_2(1,2\widehat\Tidx)$};

        \draw[shift={(6,0)},color=black] (0pt,5pt) -- (0pt,-5pt) node[below] {$\pi$};
        \draw[shift={(6,0)},color=matlabblue] (0pt,18pt) node {$\beta_1(-1,\widehat\Tidx)$};
        \draw[shift={(6,0)},color=black] (0pt,33pt) node {$\beta_3(1,2\widehat\Tidx)$};

        \draw[shift={(8,0)},color=black] (0pt,5pt) -- (0pt,-5pt) node[below] {$4\pi/3$};
        \draw[shift={(8,0)},color=matlabred]  (0pt,18pt) node {$\beta_2(1,\widehat\Tidx)$};
        \draw[shift={(8,0)},color=black] (0pt,33pt) node {$\beta_4(1,2\widehat\Tidx)$};

        \draw[shift={(10,0)},color=black] (0pt,5pt) -- (0pt,-5pt) node[below] {$5\pi/3$};
        \draw[shift={(10,0)},color=matlabblue]  (0pt,18pt) node {$\beta_2(-1,\widehat\Tidx)$} ;
        \draw[shift={(10,0)},color=black] (0pt,33pt) node {$\beta_5(1,2\widehat\Tidx)$};

        \draw[shift={(12,0)},color=black] (0pt,5pt) -- (0pt,-5pt) node[below] {$2\pi$};
    \end{tikzpicture}
    \caption{$\beta_j(\cdot, \cdot)$ for different parameters when $\widehat\Tidx=3$\vspace{-.3cm}}
    \label{fig:pd-track:anal:beta}
\end{figure}

Now, we can utili\sz{}e the expressions \cref{eq:pd-track:hs-cos}. \Cref{fig:pd-track:anal:beta} shows the spacing of the $\beta_j$s when $\widehat\Tidx=3$ for $\alpha=1$ (red) and $\alpha=-1$ (blue). Combined with \cref{eq:pd-track:hs-cos:1}, it is clear that $h_{0,0}(1,2\widehat\Tidx)$ sums over the same $\beta$s as $h_{0,0}(1,\widehat\Tidx)$ and $h_{0,0}(-1,\widehat\Tidx)$ combined. After correcting for the scaling by $\widehat\Tidx$ in \cref{eq:pd-track:hs-cos:1}, we get
\begin{equation}
    h_{0,0}(-1,\widehat\Tidx) = 2h_{0,0}(1,2\widehat\Tidx)-h_{0,0}(1,\widehat\Tidx) = \psi\frac{z_1z_2}{\varphi(z_2-z_1)}\left(\frac1{1+z_1^{\widehat\Tidx}}-\frac1{1+z_2^{\widehat\Tidx}}\right)    \mper
\end{equation}
A similar technique can be used for $h_{0,\en}$, yielding
\begin{equation}
\begin{aligned}
    h_{0,\en}(-1,\widehat\Tidx) &= h_{0,\en}(1,\widehat\Tidx)-2h_{0,\en}(1,2\widehat\Tidx) \\&= -\psi\frac{z_1z_2}{\varphi(z_2-z_1)}\left(\frac{z_1}{1+z_1^{\widehat\Tidx}}-\frac{z_2}{1+z_2^{\widehat\Tidx}}\right)    \mper
\end{aligned}
\end{equation}
As a last step, we have that $z_1z_2=1$. To see this, note that
\begin{equation}
\begin{aligned}
    z_1z_2 &= \frac1{4\varphi^2}\left(
        (1+\varphi^2+\psi^2)^2 - (\sqrt{(1+\varphi^2+\psi^2)^2-4\varphi^2})^2
    \right)\\
    &= \frac1{4\varphi^2}\left(
        (1+\varphi^2+\psi^2)^2 - (1+\varphi^2+\psi^2)^2 + 4\varphi^2
    \right) = 1    \mper
\end{aligned}
\end{equation}
Eliminating the square root is allowed due to its contents always being non-negative. Indeed, $(1+\varphi^2+\psi^2)^2-4\varphi^2$ reaches a minimum for $\psi=0$, where we get $1+2\varphi^2+\varphi^4-4\varphi^2=(1-\varphi^2)^2$, which cannot be negative. Filling the expressions for the $h$s into \cref{eq:thm:pd-track:eigs:abeq} and \cref{eq:thm:pd-track:eigs:Mredeig} proves the theorem.

\clearpage
\section{Proof of Theorem \ref*{thm:pd-tc:eigs}} \label{sec:apdx-pdtce}
    We use the notation $C(\alpha)$ for the top-left block of the matrix inverted in \cref{thm:pd-tc:eigs:M}. We start the proof similarly to \cref{sec:apdx-pdte}.
\subparagraph{Finding $M_\mathrm{red}$} The inverse of the block-triangular matrix can be rewritten as
\begin{equation}
    M = \begin{litmat}
        \overbrace{C(\alpha)^{-1}}^{\eqqcolon H} & -\psi \overbrace{C(\alpha)^{-1}C(\alpha)^{-\trsp}}^{\eqqcolon G}\\& \underbrace{C(\alpha)^{-\trsp}}_{= H^\trsp}\\
    \end{litmat}\begin{litmat}[ccc|ccc]
        &&\alpha\varphi&&&\\
        &&&&&\\
        \cmidrule(lr){1-3}\cmidrule(lr){4-6}&&&&&\\
        &&-1&\alpha\varphi&&\\
    \end{litmat}\eqqcolon \left[\begin{smallarray}{ccc|ccc}
        &&a_1&b_1&&\\
        &&\vdots&\vdots\\
        &&a_{\Tidx}&b_{\Tidx}\\
        \cmidrule(lr){1-3}\cmidrule(lr){4-6} &&a_{\Tidx+1}&b_{\Tidx+1}\\
        &&\vdots&\vdots\\
        &&a_{2\Tidx}&b_{2\Tidx}\\
    \end{smallarray}\right]
\end{equation}
such that the potentially non-zero eigenvalues of $M$ are the same as those of
\begin{equation}    \label{eq:pd-tc:anal:Mred}
    M_\mathrm{red} = \begin{litmat}
        a_\Tidx & b_\Tidx\\
        a_{\Tidx+1} & b_{\Tidx+1}\\
    \end{litmat} = \begin{litmat}
        \alpha\varphi h_{\en,0}+\psi g_{\en,\en} & -\alpha\varphi\psi g_{\en,\en}\\
        -h_{\en,0} & \alpha\varphi h_{\en,0}\\
    \end{litmat}    \mper
\end{equation}
\subparagraph{Reducing the unknowns to $H$} It seems that we need the values $h_{\en,0}$ and $g_{\en,\en}$ to make further progress. First, let us consider
\begin{equation}
G = C(\alpha)^{-1}C(\alpha)^{-\trsp} = (C(\alpha)^\trsp C(\alpha))^{-1}    \mper
\end{equation}
To solve a similar problem in \cref{thm:pd-track:eigs}'s proof, we noted that the matrix being inverted was alpha-circulant and acted on that knowledge. However, if $\abs\alpha\ne1$, this is not the case anymore (as can easily be checked), so another method needs to be found. We can first express $g_{\en,\en}$ in terms of $H$ as
\begin{equation}
    g_{\en,\en} = \mv{H_{\en,:}}\mv{(H^\trsp)_{:,\en}} = \norm{\mv{H_{\en,:}}}_2^2 = \sum\nolimits_{j=0}^{\Tidx-1}{h_{\en,j}^2}    \mper
\end{equation}

\subparagraph{Inverting $C(\alpha)$} Let us now work on the problem of finding $H=C(\alpha)^{-1}$. If $C(\alpha)$ were fully circulant, \cite[Theorem 1(d)]{searleInvertingCirculantMatrices1979} would offer a relatively simple analytical expression for its inverse; unfortunately, it is alpha-circulant. Even the technique to invert $(-1)$-circulant matrices from \cref{thm:pd-track:eigs}'s proof does not suffice here. Luckily, we have yet another trick up our sleeves.

Again, the key is inverting the diagonali\sz{}ation in \cref{eq:pd-track:alpha:Calpha-fact}. Consider doing so for $C(\alpha)$, as well as for an \emph{actually} circulant matrix $\widehat C$ -- defined later -- giving
\begin{subequations} \label{eq:pd-tc:anal:C-Chat}
\begin{align}
    C(\alpha)^{-1}&=\Gamma_\alpha^{-1}\dfm ^*\diag(\sqrt\Tidx\dfm \Gamma_\alpha\mv c_1)^{-1}\dfm \Gamma_\alpha    \mcom\\
    \widehat C^{-1}&=\dfm ^*\diag(\sqrt\Tidx\dfm \mv{\widehat c}_1)^{-1}\dfm     \mper
\end{align}
\end{subequations}
If we now require $\mv{\widehat c}_1=\Gamma_\alpha\mv c_1$, this fully defines $\widehat C$. But then \cref{eq:pd-tc:anal:C-Chat} means that $H = C(\alpha)^{-1} = \Gamma_\alpha^{-1}\widehat C^{-1}\Gamma_\alpha$, which allows computing $H$. By \cite[Theorem 1(d)]{searleInvertingCirculantMatrices1979},
\begin{equation}
    h_{\en,j} = \alpha^{(j-(\Tidx-1))/\Tidx}\alpha^{(\Tidx-j-1)/\Tidx}\frac{\varphi^{\Tidx-j-1}}{1-\alpha\varphi^\Tidx} = \frac{\varphi^{\Tidx-j-1}}{1-\alpha\varphi^\Tidx}
\end{equation}
from which immediately follow $h_{\en,0} = \frac{\varphi^{\Tidx-1}}{1-\alpha\varphi^\Tidx}$ and
\begin{equation*}
\begin{aligned}
    g_{\en,\en} &= \sum_{j=0}^{\Tidx-1}{h_{\en,j}^2} = \sum_{j=0}^{\Tidx-1}\left(\frac{\varphi^j}{1-\alpha\varphi^\Tidx}\right)^2 = \frac1{(1-\alpha\varphi^\Tidx)^2}\sum_{j=0}^{\Tidx-1}{(\varphi^2)^j} \\&= \frac1{(1-\alpha\varphi^\Tidx)^2}\frac{1-\varphi^{2\Tidx}}{1-\varphi^2}    \mper
\end{aligned}
\end{equation*}
Filling these into \cref{eq:pd-tc:anal:Mred} gives \cref{eq:lmm:pd-tc:anal:eigs:Mred}, while the limits \cref{eq:lmm:pd-tc:anal:eigs:lim} are then trivial as the entire second column of $M_\mathrm{red}$ goes to zero when $\alpha\rightarrow0$.

\clearpage
\section{Proofs of auxiliary lemmas} \label{sec:apdx-pdta}
    \begin{lemma}[Some properties of \cref{thm:pd-track:eigs}'s $z_1$ and $z_2$] \label{lmm:apdx-proof:pd-track:zs}
    Defining
    \begin{equation} \label{eq:lmm:apdx-proof:pd-track:zs:zs}
        z_{\{1,2\}}=(1+\varphi^2+\psi^2\pm\sqrt{(1+\varphi^2+\psi^2)^2-4\varphi^2})/(2\varphi)    \mcom
    \end{equation}
    the following properties hold.
    \begin{enumerate}[label=(\alph*)]
        \item If $0<\varphi<1$, both $z_1$ and $z_2$ are real-valued and it holds that $0 < z_2 < 1 < z_1$.
        \item If $0<\varphi<1$, it holds that $z_2 \le \varphi$.
    \end{enumerate}
\end{lemma}
\begin{proof}
    We prove these claims one by one.
    \begin{enumerate}[label=(\alph*)]
        \item The quantity in \cref{eq:lmm:apdx-proof:pd-track:zs:zs}'s square root reads
        \begin{equation}
            (1+\varphi^2+\psi^2)^2-4\varphi^2 = (1+\varphi^2+\psi^2 + 2\varphi)\underbrace{(1+\varphi^2+\psi^2 - 2\varphi)}_{=(1-\varphi)^2+\psi^2}    \mcom
        \end{equation}
        which is positive, such that the $z$s are real numbers. They are also positive, as follows from $4\varphi^2>0$ and $\varphi>0$. Furthermore, $z_1>z_2$. Since their product is $1$ (see \cref{sec:apdx-pdte}), $z_1$ must be larger than $1$ while $z_2$ is smaller.
        \item We write
        \begin{align*}
            &&z_2 &\le \varphi\\
            \Leftrightarrow&&1+\varphi^2+\psi^2-\sqrt{(1+\varphi^2+\psi^2)^2-4\varphi^2} &\le 2\varphi^2\\
            \Leftrightarrow&&1-\varphi^2+\psi^2 &\le \sqrt{(1+\varphi^2+\psi^2)^2-4\varphi^2}\\
            \Leftrightarrow&&(1-\varphi^2+\psi^2)^2 &\le (1+\varphi^2+\psi^2)^2-4\varphi^2\\
            \Leftrightarrow&&4\varphi^2 &\le 4\varphi^2(1+\psi^2)    \mcom
        \end{align*}
        which is clearly true.
    \end{enumerate}
\end{proof}

\begin{lemma}[Some properties of \cref{thm:pd-track:eigs}'s $\omega_1$ and $\omega_2$] \label{lmm:apdx-proof:pd-track:thetas}
    Define
    \begin{equation}
        \omega_{\{1,2\}} = \frac1{z_2-z_1}\bigg(
            \frac{z_1-\varphi\pm\psi\iu}{1 + z_1^{\widehat\Tidx}}-\frac{z_2-\varphi\pm\psi\iu}{1 + z_2^{\widehat\Tidx}}
        \bigg)
    \end{equation}
    with $z_1$ and $z_2$ as in \cref{eq:lmm:apdx-proof:pd-track:zs:zs}, where $\widehat\Tidx\ge 1$ and $0<\varphi<1$. Then denote by $\Re(\omega) = \Re(\omega_1) = \Re(\omega_2)$ the real part characterising the $\omega$s and by $\Im(\omega) = \Im(\omega_1) = -\Im(\omega_2)$ the imaginary part.
    \begin{enumerate}[label=(\alph*)]
        \item It holds that $\Re(\omega) < 0$ increases monotonically with increasing $\widehat\Tidx$.
        \item It holds that $\Im(\Omega) > 0$ increases monotonically with increasing $\widehat\Tidx$.
        \item Following (a), it holds that $-\frac12 < \Re(\omega_{\{1,2\}})$.
        \item Following (a) and (b), it holds that $\abs{\frac12 + \omega_{\{1,2\}}} < \frac12$.
    \end{enumerate}
\end{lemma}
\begin{proof}
    Once again, the claims are addressed one by one.
    \begin{enumerate}[label=(\alph*)]
        \item We rewrite $\Re(\omega) = -\frac1{z_1 - z_2}\big(
                    \frac{z_1-\varphi}{1 + z_1^{\widehat\Tidx}}+\frac{\varphi-z_2}{1 + z_2^{\widehat\Tidx}}
                \big)$
            where, due to \cref{lmm:apdx-proof:pd-track:zs}(a) and \cref{lmm:apdx-proof:pd-track:zs}(b), all numerators and denominators are positive. \Cref{sec:apdx-pdte} showed $z_1z_2=1\Leftrightarrow z_2=1/z_1$ -- filling this in, we obtain
            \begin{equation}
                \frac\dif{\dif\widehat\Tidx}\Re(\omega) = \frac1{z_1-1/z_1}\frac{z_1^{\widehat\Tidx-1}(-2\varphi z_1 + z_1^2 + 1)\log z_1}{(z_1^{\widehat\Tidx}+1)^2}    \mper
            \end{equation}
            This is always positive (recall that $z_1>1>\varphi$), such that the claim holds.

        \item A similar technique works for $\Im(\omega) = \frac\psi{z_1-1/z_1}\bigl(\frac1{1+1/z_1^{\widehat\Tidx}}-\frac1{1+z_1^{\widehat\Tidx}}\bigr)$. We find
            \begin{equation}
                \frac\dif{\dif\widehat\Tidx}\Im(\omega) = \frac\psi{z_1-1/z_1}\frac{2z_1^{\widehat\Tidx}\log z_1}{(z_1^{\widehat\Tidx}+1)^2}    \mcom
            \end{equation}
            which is a positive quantity, confirming the claim.

        \item From \cref{lmm:apdx-proof:pd-track:thetas}(a), it follows that
        \begingroup\allowdisplaybreaks
        \begin{align*}
            \Re(\omega) &\ge -\frac1{z_1 - z_2}\bigg(
                \frac{z_1-\varphi}{1 + z_1}+\frac{\varphi-z_2}{1 + z_2}
            \bigg) = \frac{(z_1-\varphi)(1+z_2)+(\varphi-z_2)(1+z_1)}{(z_2-z_1)(1+z_1)(1+z_2)}\\
            &= -\frac{z_1-z_2+\varphi(z_1-z_2)}{(z_1-z_2)(2+z_1+z_2)} = -\frac{1+\varphi}{2+z_1+z_2}\\
            &= -\frac{1+\varphi}{2+(1+\varphi^2+\psi^2)/\varphi} = -\frac{\varphi(\varphi+1)}{(\varphi+1)^2+\psi^2}    \mper
        \end{align*}
        \endgroup
        Thus
        \begin{align*}
            -\frac12 &< \Re(\omega) \Leftarrow \frac{\varphi(\varphi+1)}{(\varphi+1)^2+\psi^2} < \frac12 \Leftrightarrow \varphi^2+\varphi < \frac{\varphi^2}2+\varphi+\frac12+\frac{\psi^2}2    \mcom
        \end{align*}
        the latter of which is true from the condition $0 < \varphi < 1$.

        \item Since $-1/2 < \Re(\omega) < 0$, it holds that $\abs{1/2+\omega_{\{1,2\}}}^2$ is bounded above by the squares of $\Re(1/2+\omega_{\{1,2\}})$ maximi\sz{}ed over $\widehat\Tidx$ and $\Im(1/2+\omega)$ maximi\sz{}ed over $\widehat\Tidx$. Thus, using the fact that these maxima are attained for $\widehat\Tidx\rightarrow\infty$ (where $z_1^{\widehat\Tidx}\rightarrow\infty$ and $z_2^{\widehat\Tidx}\rightarrow0$),
        \begin{align*}
            \abs{\frac12+\omega_{\{1,2\}}}^2 &\le 
                \left(\frac12 - \frac{\varphi-z_2}{z_1-z_2}\right)^2 + \left(\frac{\psi}{z_1-z_2}\right)^2\\
                &= \left(\frac{(1+\varphi^2+\psi^2)/(2\varphi)-\varphi}{z_1-z_2}\right)^2 + \left(\frac\psi{z_1-z_2}\right)^2\\
                &= \frac{(1-\varphi^2+\psi^2)^2 + (2\varphi\psi)^2}{\left(2\sqrt{(1+\varphi^2+\psi^2)^2-4\varphi^2}\right)^2} = \frac14    \mper
        \end{align*}
        This proves the claim.
    \end{enumerate}
\end{proof}

\begin{lemma}[Approximation of $1/z$ on a semi-dis\ukus{c}{k}] \label{lmm:apdx-proof:pd-track:polappr}
    Denote by $\mathcal D_{0.5,+}$ the right half of a dis\ukus{c}{k} in the complex plane, cent\ukus{r}{er}ed at $0.5$ and with radius $0.5$. Define $R=2$. Then, for any $0<\rho<R$, there exists some constant $\kappa_\rho$ such that, for any integer $\iidx\ge0$, there exists a degree-$\iidx$ polynomial that approximates $f(z)=1/z$ on $\mathcal D_{0.5,+}$ with an infinity-norm error of at most $\kappa_\rho\rho^{-\iidx}$.
\end{lemma}
\begin{proof}
    $f$ is analytic in the complex plane, except for the origin $z=0$. According to \cite[Theorem 4.1]{saffLogarithmicPotentialTheory2010a}, we must find the unique Riemann (conformal) mapping $z\rightarrow w(z)$ of the exterior of $\mathcal D_{0.5,+}$ to the exterior of the unit dis\ukus{c}{k} $\mathcal D$ for which $w(\infty) = \infty$ and $w'(\infty)>0$. The lemma then holds for any $R$ for which $f$ can be analytically extended to the interior of the $w$-preimage $\Gamma_R$ of the radius-$R$ origin-cent\ukus{r}{er}ed circle.

    \clearpage
    In essence, we must find a conformal mapping $w$ from $\mathbb C\backslash\mathcal D_{0.5,+}$ to $\mathbb C\backslash\mathcal D$ for which $w(\infty)=\infty$ and $w'(\infty)>0$, checking how far $w(0)$ is from the origin. We construct
    \begin{equation}
        w(z) = w_5(w_4(w_3(w_2(w_1(z)))))    \mper
    \end{equation}

    \begin{figure}
        \centering
        \begin{minipage}{.25\textwidth}
            \begin{tikzpicture}[scale=.4]
                \begin{axis}[
                    axis equal image,
                    xmin=-2,
                    xmax=2,
                    ymin=-2,
                    ymax=2,
                    ticklabel style={font=\Large}
                ]

                \fill[gray!50] (-2,-2) rectangle (2,2);

                \begin{scope}
                    \clip (.5,-.6) rectangle (1.1,.6);
                    \draw (.5,0) circle(.5);
                    \fill[white] (.5,0) circle(.5);
                    \draw (.5,-.5) -- (.5,.5);
                \end{scope}
                \fill[red] (0,0) circle(.06);
                \end{axis}
                \node[above] at (2.9,5.6) {\color{white}$w_1(z)$};
                \node[above] at (2.9,5.6) {$z$};
            \end{tikzpicture}
        \end{minipage}
        \hfill
        \begin{minipage}{.25\textwidth}
            \begin{tikzpicture}[scale=.4]
                \begin{axis}[
                    axis equal image,
                    xmin=-2,
                    xmax=2,
                    ymin=-2,
                    ymax=2,
                    ticklabel style={font=\Large}
                ]

                \fill[gray!50] (-2,-2) rectangle (2,2);

                \begin{scope}
                    \clip (0,-.1) rectangle (2,2.1);
                    \draw (0,1) circle(1);
                    \fill[white] (0,1) circle(1);
                    \draw (0,0) -- (0,2);
                \end{scope}
                \fill[red] (-1,1) circle(.06);

                \end{axis}
                \node[above] at (2.9,5.6) {$w_1(z)$};
            \end{tikzpicture}
        \end{minipage}
        \hfill
        \begin{minipage}{.25\textwidth}
            \begin{tikzpicture}[scale=.4]
                \begin{axis}[
                    axis equal image,
                    xmin=-2,
                    xmax=2,
                    ymin=-2,
                    ymax=2,
                    ticklabel style={font=\Large}
                ]

                \fill[gray!50] (-2,-2) rectangle (2,2);
                \fill[white] (0,-2) rectangle (2,0);
                \draw (0,-2) -- (0,0);
                \draw (0,0) -- (2,0);

                \fill[red] (-.5,0) circle(.06);

                \end{axis}
                \node[above] at (2.9,5.6) {$w_2(w_1)$};
            \end{tikzpicture}
        \end{minipage}
        \vfill
        \begin{minipage}{.25\textwidth}
            \begin{tikzpicture}[scale=.4]
                \begin{axis}[
                    axis equal image,
                    xmin=-2,
                    xmax=2,
                    ymin=-2,
                    ymax=2,
                    ticklabel style={font=\Large}
                ]

                \fill[gray!50] (-2,-2) rectangle (2,2);
                \fill[white] (-2,-2) rectangle (2,0);
                \draw (-2,0) -- (2,0);

                \fill[red] (-.315,.5456) circle(.06);

                \end{axis}
                \node[above] at (2.9,5.6) {$w_3(w_2)$};
            \end{tikzpicture}
        \end{minipage}
        \hfill
        \begin{minipage}{.25\textwidth}
            \begin{tikzpicture}[scale=.4]
                \begin{axis}[
                    axis equal image,
                    xmin=-2,
                    xmax=2,
                    ymin=-2,
                    ymax=2,
                    ticklabel style={font=\Large}
                ]

                \fill[gray!50] (-2,-2) rectangle (2,2);

                \begin{scope}
                    \draw (0,0) circle(1);
                    \fill[white] (0,0) circle(1);
                \end{scope}
                \fill[red] (1,-1.7321) circle(.06);

                \end{axis}
                \node[above] at (2.9,5.6) {$w_4(w_3)$};
            \end{tikzpicture}
        \end{minipage}
        \hfill
        \begin{minipage}{.25\textwidth}
            \begin{tikzpicture}[scale=.4]
                \begin{axis}[
                    axis equal image,
                    xmin=-2,
                    xmax=2,
                    ymin=-2,
                    ymax=2,
                    ticklabel style={font=\Large}
                ]

                \fill[gray!50] (-2,-2) rectangle (2,2);

                \begin{scope}
                    \draw (0,0) circle(1);
                    \fill[white] (0,0) circle(1);
                \end{scope}
                \fill[red] (-2,0) circle(.1);

                \end{axis}
                \node[above] at (2.9,5.6) {$w_5(w_4)$};
            \end{tikzpicture}
        \end{minipage}

        \caption{
            Conformal maps forming $w$. The red dots follow $f$'s pole from $\mathcal D_{0.5,+}$ to $\mathcal D$. The horizontal and vertical axes denote the real and imaginary parts, respectively.
            \vspace{-.5cm}
        }
        \label{fig:apdx-proof:pd-track:conform}
    \end{figure}

    First, $w_1$ takes $\mathcal D_{0.5,+}$, moves it with its bottom corner to the origin and magnifies it by a factor of two. The mapping that accomplishes this is $w_1 = 2z-1+\iu$. Then $w_2$ maps the exterior of the semi-dis\ukus{c}{k} into three quadrants. This can be done by the mapping $w_2 = 1/w_1+\iu/2$. Next, $w_3$ collapses three quadrants into a half-plane with the mapping $w_3 = w_2^{2/3}$. We can then turn a half-plane into the exterior of the unit dis\ukus{c}{k} through a M\"obius transformation of the form $w_4 = \frac{w_3-\beta^*}{w_3-\beta}$ for some $\beta$. Recall that $w$ should map $\infty$ to $\infty$; this can be done by taking $\beta$ to be the image of $\infty$ up until now. If $z=\infty$, we obtain $w_1=\infty$, $w_2=\iu/2$, and $w_3=(\iu/2)^{2/3}$. So setting $\beta=(\iu/2)^{2/3}$, $w_4$ is now determined. Finally, we find $w_4'(z=\infty) = (-3\sqrt3+9\iu)/4$, so with $w_5 = \exp(-2\pi\iu/3)w_4$ we end up with $w'(\infty) = (3\sqrt3)/2>0$.
    
    \Cref{fig:apdx-proof:pd-track:conform} illustrates the mapping $w$. The pole at $z=0$ maps to $w(0)=-2$, which is at distance $R=2$ from the origin. This concludes the proof.
\end{proof}

\section*{Acknowledgments}
    We are grateful to Ignace Bossuyt, Giovanni Conni, Toon Ingelaere, and Vince Maes for their thorough reviews and helpful comments. We also thank the anonymous referees for providing valuable feedback and suggestions, which greatly improved the quality of the paper.

\bibliographystyle{siamplain}
\bibliography{references}
    
\end{document}